\documentclass[10pt]{amsart}
\usepackage{amssymb}
\usepackage{amsmath,amssymb,amsfonts,amsthm,graphics,
latexsym, amscd, amsfonts, epsfig, eepic,epic,psfrag}
\usepackage{mathrsfs}
\usepackage{color}
\usepackage{eucal}
\usepackage{eufrak}
\usepackage[all]{xypic}
\usepackage{xspace}

\theoremstyle{plain}
\newtheorem{theorem}{Theorem}
\newtheorem{corollary}{Corollary}

\newtheorem{lemma}{Lemma}
\newtheorem{proposition}{Proposition}

\theoremstyle{definition}

\theoremstyle{example}

\newtheorem{remark}{Remark}
\theoremstyle{remark}

\numberwithin{equation}{section}

\begin{document}

\title[A generalization of the Brauer algebra]
      {A generalization of the Brauer algebra}
\author{William Y.C. Chen and Christian M. Reidys$^{\,\star}$}
\address{Center for Combinatorics, LPMC-TJKLC 
           \\
         Nankai University  \\
         Tianjin 300071\\
         P.R.~China\\
         Phone: *86-22-2350-6800\\
         Fax:   *86-22-2350-9272}
\email{reidys@nankai.edu.cn}
\thanks{}
\keywords{}
\date{February, 2009}
\begin{abstract}
We study two variations of the Brauer algebra $\mathbb{B}_n(x)$. 
The first is the algebra $\mathbb{A}_n(x)$, which generalizes the
Brauer algebra by considering loops. The second is the algebra 
$\mathbb{L}_n(x)$, the $\mathbb{A}_n(x)$-subalgebra generated by 
diagrams without horizontal arcs. 
$\mathbb{A}_n(x)$ and $\mathbb{L}_n(x)$ exhibit for $x\neq 0$ an
hereditary-chain indexed by all integers. Following the ideas of
Martin~\cite{Martin} in the context of the partition algebra, and
Doran {\it et al.} \cite{Doran} for the Brauer algebra, we study
semisimplicity of $\mathbb{A}_n(x)$ using restriction and induction 
in $\mathbb{A}_n(x)$ and $\mathbb{L}_n(x)$. 
Our main result is that $\mathbb{A}_n(x)$ is semisimple if 
$x\not\in\mathbb{Z}$ and that $\mathbb{L}_n(x)$ is semisimple if 
$x\neq 0$.
\end{abstract}
\maketitle
{{\small
}}

\section{Introduction}\label{S:Introduction}

In this paper, we study the semisimplicity of the two diagram algebras 
$\mathbb{A}_{n}(x)$ and $\mathbb{L}_n(x)$. $\mathbb{A}_n(x)$ generalizes 
the Brauer algebra, $\mathbb{B}_n(x)$, by containing diagrams in which 
vertices can be incident to loops (or equivalently, isolated vertices). 
$\mathbb{L}_n(x)$ is the $\mathbb{A}_n(x)$-subalgebra generated by all 
diagrams without any horizontal arcs.
The motivation for considering these algebras is twofold: on the one hand 
in the context of Schur-Weyl duality: $\mathbb{A}_n(x)$ 
is the centralizer algebra of the group of stochastic, orthogonal matrices 
and $\mathbb{L}_n(x)$ is the centralizer algebra of the group of stochastic, 
invertible matrices.
On the other hand, $\mathbb{A}_n(x)$ is as the algebra of partial matchings 
of importance for RNA pseudoknot structures, i.e.~helical configurations of 
RNA primary sequences with cross-serial nucleotide interactions 
\cite{Reidys:07pseu}.

The Brauer (centralizer) algebras \cite{Br} over the field $F=K(x)$,
denoted by $\mathbb{B}_{n}(x)$, are finite dimensional $F$-algebras 
indexed by a positive integer $n$ and $x$, which is either algebraic 
or transcendent over $K$. $\mathbb{B}_{n}(x)$ is the centralizer
algebra for the orthogonal or symplectic group on the
$n$th tensor powers of the natural representation. 
$\mathbb{B}_{n}(x)$ has been studied by various authors mainly using
combinatorial methods, see \cite{BE,Bw,EK,HW1,HW2,HW3} and
\cite{Sundaram}. Hanlon and Wales conjectured that $\mathbb{B}_{n}(x)$ 
is semisimple for all
$x\not\in\mathbb{Z}$ \cite{HW1}. Their conjecture was proved by
Wenzl~{\cite{Wen}} and Rui~\cite{Rui1} gave necessary and sufficient
conditions for the Brauer algebras to be semisimple. 

The analysis presented here is based on the concepts of Martin~\cite{Martin} 
developed in the context of the partition algebra, $\mathbb{P}_n$. 
Martin's key idea was to relate the 
existence of certain embeddings to semisimplicity. 
Subsequently, Doran {\it et al.}~{\cite{Doran}} used this framework in order
to offer an alternative to Wenzel's proof of semisimplicity. 
Wenzl's inductive construction hinges on an interpretation of a key ideal 
in $\mathbb{B}_{n}(x)$ as the tensor product $\mathbb{B}_{n-1}(x)
\otimes_{\mathbb{B}_{n-2}(x)}\mathbb{B}_{n}(x)$ \cite{Jones} and the 
nondegeneracy of a Markov-trace arising naturally in the construction of 
the latter. The nondegeneracy of this trace form is a result of Weyl's 
character formulas and is in this sense somewhat ``unsatisfactory''. The 
work of Martin~\cite{Martin} and Doran {\it et al.}~{\cite{Doran}} puts 
semisimplicity in the context of quasi-hereditarity and allows to avoid
the use of Markov-traces.

Let $\mathscr{A}_n$ be the set of partial $1$-factors over $2n$ vertices,
i.e.~graphs over $2n$ vertices in which each vertex has either degree one 
zero. We refer to $\mathscr{A}_n$-elements as diagrams and represent them 
by arranging the $2n$ vertices in two rows, each containing $n$ vertices, 
with the rows arranged one on top of the other. Furthermore, we equip each 
isolated vertex with a loop. The $n$ 
top-vertices are labeled by $[n]=\{1,\dots,n\}$ in increasing order and 
the $n$ bottom-vertices are labeled by
$[n']=\{1',\dots,n'\}$. 
Let $\mathscr{L}_n\subset\mathscr{A}_n$
be the subset of all $\mathscr{A}_n$-diagrams without any horizontal
arcs. We denote the subset of $\mathscr{L}_n$-diagrams having only vertical 
arcs by $\mathscr{S}_{n}$.
When drawing diagrams, we oftentimes omit
vertex labels. For instance,
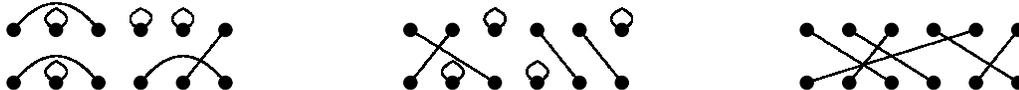
\begin{center}
\begin{picture}(390,0)
\put(0,0){\circle*{5}} \put(16,0){\circle*{5}}
\put(32,0){\circle*{5}} \put(48,0){\circle*{5}}
\put(64,0){\circle*{5}}\put(80,0){\circle*{5}}
\put(150,0){\circle*{5}} \put(166,0){\circle*{5}}
\put(182,0){\circle*{5}} \put(198,0){\circle*{5}}
\put(214,0){\circle*{5}}\put(230,0){\circle*{5}}
\put(300,0){\circle*{5}}\put(316,0){\circle*{5}}
\put(332,0){\circle*{5}} \put(348,0){\circle*{5}}
\put(364,0){\circle*{5}}\put(380,0){\circle*{5}}
\put(0,20){\circle*{5}} \put(16,20){\circle*{5}}
\put(32,20){\circle*{5}} \put(48,20){\circle*{5}}
\put(64,20){\circle*{5}}\put(80,20){\circle*{5}}
\put(150,20){\circle*{5}} \put(166,20){\circle*{5}}
\put(182,20){\circle*{5}} \put(198,20){\circle*{5}}
\put(214,20){\circle*{5}} \put(230,20){\circle*{5}}
\put(300,20){\circle*{5}}\put(316,20){\circle*{5}}
\put(332,20){\circle*{5}} \put(348,20){\circle*{5}}
\put(364,20){\circle*{5}}\put(380,20){\circle*{5}}
\qbezier[400](0,20)(16,40)(32,20) \qbezier[400](0,0)(16,20)(32,0)
\qbezier[400](48,0)(64,20)(80,0)
\qbezier[40](16,0)(8,4)(16,8)\qbezier[40](16,0)(24,4)(16,8)
\qbezier[40](16,20)(8,24)(16,28)\qbezier[40](16,28)(24,24)(16,20)
\qbezier[40](64,20)(56,24)(64,28)\qbezier[40](64,20)(72,24)(64,28)
\qbezier[40](48,20)(40,24)(48,28)\qbezier[40](48,20)(56,25)(48,28)
\qbezier(80,20)(72,10)(64,0)\qbezier(182,20)(174,24)(182,28)
\qbezier(182,20)(190,24)(182,28) \qbezier(230,20)(222,24)(230,28)
\qbezier(230,20)(238,24)(230,28)\qbezier(166,0)(158,4)(166,8)
\qbezier(166,0)(174,4)(166,8)\qbezier(198,0)(190,4)(198,8)
\qbezier(198,0)(206,4)(198,8)\qbezier(150,20)(166,10)(182,0)
\qbezier(166,20)(158,10)(150,0)\qbezier(198,20)(206,10)(214,0)
\qbezier(214,20)(222,10)(230,0)\qbezier(364,20)(332,10)(300,0)
\qbezier(300,20)(316,10)(332,0)\qbezier(316,20)(332,10)(348,0)
\qbezier(348,20)(364,10)(380,0)\qbezier(332,20)(324,10)(316,0)
\qbezier(380,20)(372,10)(364,0)
\end{picture}
\end{center}
are particular $\mathscr{A}_n$-, $\mathscr{L}_n$- and
$\mathscr{S}_n$-diagrams. By abuse of notation, we write
$S_n$ instead of $\mathscr{S}_n$, identifying $S_n$ with its embedding
into $\mathscr{A}_n$. As for their cardinalities we immediately compute
\begin{equation}
\vert \mathscr{A}_n\vert =\sum_{j=0}^n\binom{2n}{2j}
\prod_{i=0}^{(n-j)-1}(2(n-j)-1-2i)
\quad\text{\rm and}\quad
\vert \mathscr{L}_n\vert =\sum_{j=0}^n\binom{n}{j}\binom{n}{j}\,(n-j)!,
\end{equation}
where the factor $\prod_{i=0}^{(n-j)-1}(2(n-j)-1-2i)$ equals the
dimension of the Brauer-algebra $\mathbb{B}_{n-j}(x)$. Arcs joining two
different vertices, contained both in the top or bottom row are
called horizontal arcs. Arcs joining top- and bottom-vertices are
called vertical arcs. The induced subgraph of the top and bottom row
of a diagram $\mathfrak{a}$ is denoted by $\text{\rm
top}(\mathfrak{a})$ and $\text{\rm bot}(\mathfrak{a})$. Let
$\mathfrak{e}_{i}$ be the diagram having straight verticals except
of the horizontal arcs connecting $i,i+1$ and $i',(i+1)'$,
$\mathfrak{u}_{i}$ having straight verticals and loops at $i$ and
$i'$ and $\mathfrak{g}_{i}$ having straight verticals except of the
vertical arcs $(i,(i+1)')$ and $(i+1,i')$. Pictorially,
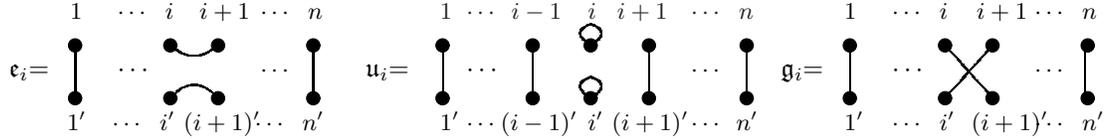
\begin{center}
\begin{picture}(20,50)
\put(-37,30){\small{$1$}}\put(-27,30){\small{$\cdots$}}
\put(-10,30){\small{$i-1$}}\put(19,30){\small{$i$}}\put(30,30){\small{$i+1$}}
\put(58,30){\small{$\cdots$}}\put(76,30){\small{$n$}}
\put(-36,0){\circle*{5}} \put(-2,0){\circle*{5}}
\put(20,0){\circle*{5}} \put(42,0){\circle*{5}}
 \put(79,0){\circle*{5}}
\put(-36,20){\circle*{5}} \put(-2,20){\circle*{5}}
\put(20,20){\circle*{5}} \put(42,20){\circle*{5}}
\put(79,20){\circle*{5}}
\put(-37,-12){\small{$1'$}}\put(-27,-12){\small{$\cdots$}}
\put(-14,-12){\small{$(i-1)'$}}\put(19,-12){\small{$i'$}}
\put(28,-12){\small{$(i+1)'$}}
\put(58,-12){\small{$\cdots$}}\put(75,-12){\small{$n'$}}
\put(-65,8){$\mathfrak{u}_{i}$=} \put(-27,8){$\cdots$}
\put(58,8){\small{$\cdots$}}
\qbezier[40](-36,20)(-36,10)(-36,0) \qbezier[50](42,20)(42,10)(42,0)
\qbezier[50](79,0)(79,10)(79,20)\qbezier[40](-2,20)(-2,10)(-2,0)
\qbezier[40](20,0)(12,5)(20,8)\qbezier[40](20,8)(28,5)(20,0)
\qbezier[40](20,20)(12,25)(20,28)\qbezier[40](20,28)(28,25)(20,20)

\put(116,30){\small{$1$}}\put(134,30){\small{$\cdots$}}
\put(152,30){\small{$i$}}\put(166,30){\small{$i+1$}}
\put(188,30){\small{$\cdots$}}\put(206,30){\small{$n$}}
\put(118,0){\circle*{5}} 
\put(154,0){\circle*{5}} \put(172,0){\circle*{5}}
\put(208,0){\circle*{5}}
\put(118,20){\circle*{5}} 
\put(154,20){\circle*{5}} \put(172,20){\circle*{5}}
\put(208,20){\circle*{5}}
\put(116,-12){\small{$1'$}}\put(134,-12){\small{$\cdots$}}
\put(152,-12){\small{$i'$}}\put(164,-12){\small{$(i+1)'$}}
\put(188,-12){\small{$\cdots$}}\put(206,-12){\small{$n'$}}
\put(92,8){$\mathfrak{g}_{i}$=}\put(134,8){$\cdots$}
\put(188,8){\small{$\cdots$}}
\qbezier[40](118,20)(118,10)(118,0)
\qbezier[50](154,0)(163,10)(172,20)
\qbezier[50](172,0)(163,10)(154,20)\qbezier[50](207,0)(207,10)(207,20)

\put(-177,30){\small{$1$}}\put(-159,30){\small{$\cdots$}}
\put(-141,30){\small{$i$}}\put(-128,30){\small{$i+1$}}
\put(-105,30){\small{$\cdots$}}\put(-87,30){\small{$n$}}
\put(-175,0){\circle*{5}} 
\put(-139,0){\circle*{5}} \put(-121,0){\circle*{5}}
\put(-85,0){\circle*{5}}
\put(-175,20){\circle*{5}} 
\put(-139,20){\circle*{5}} \put(-121,20){\circle*{5}}
\put(-85,20){\circle*{5}}
\put(-178,-12){\small{$1'$}}\put(-161,-12){\small{$\cdots$}}
\put(-143,-12){\small{$i'$}}\put(-134,-12){\small{$(i+1)'$}}
\put(-107,-12){\small{$\cdots$}}\put(-89,-12){\small{$n'$}}
\put(-200,8){$\mathfrak{e}_{i}$=}\put(-159,8){$\cdots$}
\put(-105,8){\small{$\cdots$}}
\qbezier[40](-175,20)(-175,10)(-175,0)\qbezier[50](-139,0)(-130,10)(-121,0)
\qbezier[50](-139,20)(-130,12)(-121,20)\qbezier[50](-85,0)(-85,10)(-85,20)
\end{picture}
\end{center}
We now describe the multiplication of two diagrams. Let $x$ be
either a $K$-transcendent or algebraic element. We consider 
$F[\mathscr{A}_n]$, the free $F$-module generated by $\mathscr{A}_n$
and show that $F[\mathscr{A}_n]$
is a monoid whose multiplication extends that of $\mathbb{B}_n(x)$ in a
natural way. To this end let
$\mathfrak{a},\mathfrak{b}\in\mathscr{A}_n$. Let
$G(\mathfrak{a},\mathfrak{b})$ be the graph obtained by arranging
the diagram $\mathfrak{a}$ above $\mathfrak{b}$ and introducing the
verticals arcs $(i,i')$, $1\le i\le n$ where $i$ and $i'$ are
contained in $\text{\rm top}(\mathfrak{a})$ and $\text{\rm
bot}(\mathfrak{b})$-vertex, respectively. For instance,
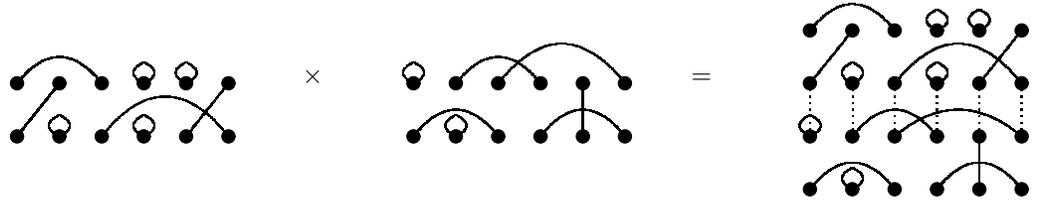
\begin{center}
\begin{picture}(380,40)\label{E:ab}
\put(0,20){\circle*{5}} \put(16,20){\circle*{5}}
\put(32,20){\circle*{5}} \put(48,20){\circle*{5}}
\put(64,20){\circle*{5}}\put(80,20){\circle*{5}}
\put(150,20){\circle*{5}}\put(166,20){\circle*{5}}
\put(182,20){\circle*{5}}\put(198,20){\circle*{5}}
\put(214,20){\circle*{5}} \put(230,20){\circle*{5}}
\put(300,20){\circle*{5}} \put(316,20){\circle*{5}}
\put(332,20){\circle*{5}} \put(348,20){\circle*{5}}
\put(364,20){\circle*{5}}\put(380,20){\circle*{5}}
\put(300,-20){\circle*{5}} \put(316,-20){\circle*{5}}
\put(332,-20){\circle*{5}} \put(348,-20){\circle*{5}}
\put(364,-20){\circle*{5}}\put(380,-20){\circle*{5}}
\put(0,0){\circle*{5}} \put(16,0){\circle*{5}}
\put(32,0){\circle*{5}} \put(48,0){\circle*{5}}
\put(64,0){\circle*{5}}
\put(80,0){\circle*{5}}\put(150,0){\circle*{5}}
\put(166,0){\circle*{5}}\put(182,0){\circle*{5}}
\put(198,0){\circle*{5}} \put(214,0){\circle*{5}}
\put(230,0){\circle*{5}} \put(300,0){\circle*{5}}
\put(316,0){\circle*{5}} \put(332,0){\circle*{5}}
\put(348,0){\circle*{5}}
\put(364,0){\circle*{5}}\put(380,0){\circle*{5}}
 \put(300,40){\circle*{5}}
\put(316,40){\circle*{5}} \put(332,40){\circle*{5}}
\put(348,40){\circle*{5}}
\put(364,40){\circle*{5}}\put(380,40){\circle*{5}}
\put(108,20){$\times$} \put(255,20){$=$}
\qbezier[400](0,20)(16,40)(32,20)
\qbezier[40](48,20)(40,24)(48,28)\qbezier[40](48,28)(56,24)(48,20)
\qbezier[40](64,20)(56,24)(64,28)\qbezier[40](64,28)(72,24)(64,20)
\qbezier[400](0,0)(8,10)(16,20)\qbezier[40](16,0)(8,4)(16,8)
\qbezier[40](16,8)(24,4)(16,0)\qbezier[200](32,0)(56,30)(80,0)
\qbezier(80,20)(72,10)(64,0)
\qbezier[40](48,0)(40,4)(48,8)\qbezier[40](48,8)(56,4)(48,0)
\qbezier[40](150,20)(142,24)(150,28)\qbezier[40](150,28)(158,24)(150,20)
\qbezier[200](166,20)(182,40)(198,20)\qbezier[200](182,20)(206,50)(230,20)
\qbezier[200](150,0)(166,20)(182,0)\qbezier[200](198,0)(214,20)(230,0)
\qbezier[40](166,0)(158,4)(166,8)\qbezier[40](166,8)(174,4)(166,0)
\qbezier(214,20)(214,10)(214,0)\qbezier[8](300,0)(300,10)(300,20)
\qbezier[8](316,0)(316,10)(316,20)\qbezier[8](332,0)(332,10)(332,20)
\qbezier[8](348,0)(348,10)(348,20)\qbezier[8](364,0)(364,10)(364,20)
\qbezier[8](380,0)(380,10)(380,20)\qbezier(300,-20)(316,0)(332,-20)
\qbezier(316,-20)(308,-16)(316,-12)\qbezier(316,-20)(324,-16)(316,-12)
\qbezier(348,-20)(364,0)(380,-20)\qbezier(364,0)(364,-10)(364,-20)
\qbezier(300,0)(292,4)(300,8)\qbezier(300,0)(308,4)(300,8)
\qbezier(316,0)(332,20)(348,0)\qbezier(332,0)(356,20)(380,0)
\qbezier(300,40)(316,60)(332,40)\qbezier(348,40)(340,44)(348,48)
\qbezier(348,40)(356,44)(348,48)\qbezier(364,40)(356,44)(364,48)
\qbezier(364,40)(372,44)(364,48)\qbezier(380,40)(372,30)(364,20)
\qbezier(316,40)(308,30)(300,20)\qbezier(316,20)(308,24)(316,28)
\qbezier(316,20)(324,24)(316,28)\qbezier(348,20)(340,24)(348,28)
\qbezier(348,20)(356,24)(348,28)\qbezier(332,20)(356,50)(380,20)
\end{picture}
\end{center}
$G(\mathfrak{a},\mathfrak{b})$ contains two types of information:
(i) $\ell(\mathfrak{a},\mathfrak{b})$, the number of
$G(\mathfrak{a},\mathfrak{b})$ components that do not contain any
vertices of $\text{\rm top}(\mathfrak{a})$ or $\text{\rm bot}(\mathfrak{b})$
and (ii) ${G'}(\mathfrak{a},\mathfrak{b})$, the graph over the
$\text{\rm top}(\mathfrak{a})$ and $\text{\rm bot}(\mathfrak{b})
$-vertices obtained as follows: any two vertices are connected by
an arc if and only if they are connected by a
$G(\mathfrak{a},\mathfrak{b})$-path.
Accordingly, we have $\mathfrak{a}\,\cdot\, \mathfrak{b}=
x^{\ell(\mathfrak{a},\mathfrak{b})} \, G'(\mathfrak{a},\mathfrak{b}) 
$ and we shall write $\mathfrak{a} \mathfrak{b}$ instead of 
$\mathfrak{a}\cdot \mathfrak{b}$. 
$F[\mathscr{A}_n]$ becomes via ``$\,\cdot\,$'' an associative, unitary 
$F$-subalgebra of the partition algebra, which we denote by $\mathbb{A}_n(x)$. 
Furthermore, via ``$\,\cdot\,$'', $F[\mathscr{L}_n]$ becomes an associative 
$F$-subalgebra of $\mathbb{A}_n(x)$, denoted by $\mathbb{L}_n(x)$. 

By abuse of notation, we write
$\mathbb{A}_n=\mathbb{A}_n(x)$, $\mathbb{B}_n=\mathbb{B}_n(x)$ and 
$\mathbb{L}_n=\mathbb{L}_n(x)$.
Furthermore, we shall assume that $F$ is a field of characteristic
zero and the term ``semisimple'' is synonymous to ``direct sum of
full matrix algebras''. In other words, $F$ is a splitting 
field of $\mathbb{A}_n$ and $\mathbb{L}_n$.

\begin{remark}
{\rm Let $\ell_1(\mathfrak{a},\mathfrak{b})$ and $\ell_2(\mathfrak{a},
\mathfrak{b})$ denote the number of inner components that are cycles and
lines with loops at the start and endpoint. Setting
\begin{equation}\label{E:mult2}
\mathfrak{a}\,\circ\, \mathfrak{b}=
x_1^{\ell_1(\mathfrak{a},\mathfrak{b})} \,
x_2^{\ell_2(\mathfrak{a},\mathfrak{b})}\,
G'(\mathfrak{a},\mathfrak{b}),
\end{equation}
we observe that $F[\mathscr{A}_n]$ becomes via ``$\circ$'' an associative 
unitary $F$-algebra, which we denote by $\mathbb{A}_n(x_1,x_2)$. Obviously,
in case of $x_1=x_2$ the multiplications ``$\circ$'' and ``$\,\cdot\,$'' 
coincide.
}
\end{remark}
As it is the case for $\mathbb{B}_n$, there exist natural embedding
between $\mathbb{A}_{n-1}$ and $\mathbb{A}_n$ obtained by adding the
vertices $n$ and $n'$ together with the straight vertical arc,
$(n,n')$, $\epsilon_n\colon \mathbb{A}_{n-1}\longrightarrow
\mathbb{A}_n$. By restriction the latter induces an embedding of 
$\mathbb{L}_{n-1}$ into $\mathbb{L}_{n}$, which we denote again by
$\epsilon_n\colon \mathbb{L}_{n-1}\longrightarrow \mathbb{L}_n$.
Furthermore, there exists an involution on $\mathbb{A}_n$ and
$\mathbb{L}_n$ obtained by transposing the rows, denoted by
$\mathfrak{a}\mapsto \mathfrak{a}^*$. We set
$\mathscr{A}_n^{m}\subset \mathscr{A}_n^{n}=\mathscr{A}_n$ to be the
subset of diagrams having at most $m$ vertical arcs and let
$\mathbb{A}_{n}^m$ be the ideal generated by $\mathscr{A}_n^{m}$.
The ideals $\mathbb{A}_{n}^m$ for $0\le m\le n$ give a filtration of
$\mathbb{A}_n$, i.e.~we have
\begin{equation}
\mathbb{A}_n^0\subsetneq \mathbb{A}_n^1\subsetneq \dots \subsetneq
\mathbb{A}_n^{n-1}\subsetneq \mathbb{A}_n^n=\mathbb{A}_n.
\end{equation}
Furthermore, let
$\mathbb{I}_n^m=\mathbb{A}_n^{m}/\mathbb{A}_n^{m-1}$ denote the
algebra induced by $\mathbb{A}_n$, which is generated by the set all
$\mathscr{A}_n$-diagrams with exactly $m$ vertical arcs, denoted by
$\mathscr{I}_n^m$. That is, we have $[\mathfrak{a}]\cdot
[\mathfrak{b}]=[\mathfrak{a}\cdot \mathfrak{b}]$ where
$[\mathfrak{a}\cdot\mathfrak{b}]$ is zero if it contains less than
$m$ vertical arcs. Similarly, we have $[\mathfrak{a}]\circ
[\mathfrak{b}]=[\mathfrak{a}\circ \mathfrak{b}]$ in case of
``$\circ$". By abuse of notation we shall identify
$[\mathfrak{a}]$ with $\mathfrak{a}$. Note that $\mathbb{I}_{n}^{n}$
is isomorphic to the group algebra $K[S_{n}]$. Similarly,
$\mathbb{L}_n$ has the filtration
\begin{equation}
\mathbb{L}_n^0\subsetneq \mathbb{L}_n^1\subsetneq \dots \subsetneq
\mathbb{L}_n^{n-1}\subsetneq \mathbb{L}_n^n=\mathbb{L}_n
\end{equation}
and by abuse of notation we denote the quotients $\mathbb{L}_n^{m}/
\mathbb{L}_n^{m-1}$ and the set all $\mathscr{L}_n$-diagrams with
exactly $m$ vertical arcs again by $\mathbb{I}_n^m$ and $\mathscr{I}_n^m$,
respectively.

An integer partition $\lambda=(\lambda_1,\lambda_2,\dots,\lambda_{s})$, 
where $\lambda_1\geq \lambda_2\geq\dots\geq \lambda_n$ is a weakly 
decreasing sequence of positive integers. 
If $\sum_i\lambda_i=n$, we write $\lambda\vdash n$. Since
any irreducible $S_{n}$-module is indexed by a partition \cite{Sagan} $\lambda$
we write them as $S^\lambda$. The dimension of $S^{\lambda}$ is
denoted by $f^{\lambda}$ and its character by $\chi^{\lambda}$. The
integers $\lambda_i$ is called the parts of $\lambda$. The Ferrer
diagram associated with a partition $\lambda$ is a collection of
boxes, $[\lambda]$, in $\mathbb{Z}^{2}$ using matrix-style
coordinates. 
The boxes are arranged in left-justified rows with
weakly decreasing numbers of boxes in each row.
For a box $p=(i,j)$ in $[\lambda]$, $j-i$ is the content of $p$,
denoted by $c(p)$. If $\lambda$ and $\mu$ are two partitions such 
that $\lambda_{i}\geq\mu_{i}$ for all $i$, then we say $\lambda$ 
contains $\mu$ and write $\mu\subseteq\lambda$. If
$\mu\subseteq\lambda$, then the skew partition $\lambda/\mu$ is the
set $[\lambda]/[\mu]$. A special case is when $\lambda/\mu$ contains
one box only, denoted by $\lambda\sqsupset \mu$. If we identify
$\lambda$ with a Ferrer diagram, then an inner corner of $\lambda$
is a node $(i,j)\in \lambda$ whose removal leaves the Ferrers
diagram of a partition. Any partition $\mu_1$ obtained by such a
removal is denoted by $\mu_1\sqsubset\lambda$. An outer corner of
$\lambda$ is a node $(i,j)\notin \lambda$ whose addition produces
the Ferrer diagram of a partition. Any partition $\mu_2$ obtained
by such an addition is denoted by $\lambda\sqsubset\mu_2$. Let
$\text{\rm res}_{S_{n-1}}^{S_{n}}S^{\lambda}$ and $\text{\rm
ind}_{S_{n}}^{S_{n+1}}S^{\lambda}$ denote the restriction and the
induced representation of $S^{\lambda}$. Then we have \cite{Sagan}
\begin{equation}\label{E:branch1}
\text{\rm res}_{S_{n-1}}^{S_{n}}S^{\lambda}\cong\bigoplus_{\mu_1
\sqsubset\lambda}S^{\mu_1}\quad\text{\rm and}\quad
\text{\rm ind}_{S_{n}}^{S_{n+1}}S^{\lambda}\cong
\bigoplus_{\lambda\sqsubset\mu_2}S^{\mu_2}.
\end{equation}
We proceed by describing the induced representation  \cite{Serre} in a 
specific way.
For any $1\le j\le n-t$, set $\tau_j=(j,n+1-t)$ and $\tau_{n+1-t}=1$.
Then $\{\tau_r\mid 1\le j\le n+1-t\}$ is a set of representatives of
$S_{n+1-t}/S_{n-t}$ and 
\begin{equation}
\text{\rm ind}_{S_{n-t}}^{S_{n+1-t}}S^\lambda\cong K[S_{n+1-t}]
\otimes_{K[S_{n-t}]}S^\lambda\cong
\bigoplus_{j=1}^{n+1-t}(S^\lambda,j).
\end{equation}
Here, the $S_{n+1-t}$-action on $\text{\rm
ind}_{S_{n-t}}^{S_{n+1-t}} S^\lambda$ is given as follows: for given
$\sigma\in S_{n+1-t}$ and $1\le j\le n+1-t$, let $s$ be the unique
index such that $\sigma\tau_j\in \tau_sS_{n-t}$ holds, then
\begin{equation}\label{E:ind-act}
\sigma\cdot (w,j)=((\tau_s^{-1}\sigma\tau_j)w,s).
\end{equation}
In the following, let $\mathbb{X}_n$ denote either $\mathbb{A}_n$ or 
$\mathbb{L}_n$. Let $M$ be a
$\mathbb{X}_n$-left module. Then $\text{\rm res}_{n-1}(M)$ denotes
the $\mathbb{X}_{n-1}$-left module, obtained via restriction with
respect to the embedding $\epsilon_n\colon\mathbb{X}_{n-1}
\longrightarrow \mathbb{X}_n$ and
$
\text{\rm ind}_{n+1}(M)=\mathbb{X}_{n+1}\otimes_{\mathbb{X}_{n}}M
$
denotes the induced $\mathbb{X}_{n+1}$-left module.

\section{$\mathbb{X}_n$-Modules}
The semisimplicity of $\mathbb{X}_n$ is closely tied to the 
structure of $\mathbb{X}_n$-modules. Therefore we shall begin by
establishing their basic properties. The latter are a result of the 
general machinery derived from the fact that $\mathbb{A}_n$ and 
$\mathbb{L}_n$ are for $x\neq 0$ quasi-hereditary algebras. However, 
we shall prove them directly. 
Let $\mathfrak{u}_{n,t}$ denote the diagram having straight verticals
except of loops incident to $(n-t+1),\cdots,n$ and
$(n-t+1)',\dots,n'$, respectively. Pictorially,
\begin{center}
\begin{picture}(60,20)
\put(0,0){\circle*{5}} \put(32,0){\circle*{5}}
\put(48,0){\circle*{5}} \put(80,0){\circle*{5}}
\put(0,20){\circle*{5}} \put(32,20){\circle*{5}}
\put(48,20){\circle*{5}} \put(80,20){\circle*{5}}
\qbezier(0,0)(0,10)(0,20)\qbezier(32,0)(32,10)(32,20)
\qbezier(48,0)(40,4)(48,8)\qbezier(48,0)(56,4)(48,8)
\qbezier(48,20)(40,24)(48,28)\qbezier(48,20)(56,24)(48,28)
\qbezier(80,0)(72,4)(80,8)\qbezier(80,0)(88,4)(80,8)
\qbezier(80,20)(72,24)(80,28)\qbezier(80,20)(88,24)(80,28)
\put(10,10){$\cdots$} \put(60,10){$\cdots$} \put(-2,30){\small{$1$}}
\put(16,30){\small{$n-t$}}\put(40,30){\small{$n-t+1$}}\put(78,30){\small{$n$}}
\put(-40,10){$\mathfrak{u}_{n,t}=$}
\end{picture}
\end{center}
Let $x\neq 0$ and $\lambda\vdash (n-t)\le n$ be a partition, we set
\begin{equation}
\mathscr{M}_{\mathbb{X}_n}(\lambda)=\mathbb{I}_{n}^{n-t}\mathfrak{u}_{n,t}
\otimes_{S_{n-t}}S^\lambda
\quad \text{and }\quad
\mathscr{N}_{\mathbb{X}_n}(\lambda)=\{w\in\mathscr{M}_{\mathbb{X}_n}
(\lambda)\mid \mathbb{I}_{n}^{n-t}w=0\}.
\end{equation}
$\mathscr{M}_{\mathbb{X}_n}(\lambda)$ and $\mathscr{N}_{\mathbb{X}_n}
(\lambda)$ become via linear extension of the action
\begin{equation}
\mathfrak{b}\cdot (\mathfrak{a}\otimes v)=
(\mathfrak{b}\mathfrak{a})\otimes v,
\end{equation}
$\mathbb{X}_n$- and $\mathbb{I}_n^{n-t}$-left modules, respectively. 
Indeed, for any
$0\le t\le n$, $\mathbb{X}_{n}^{n-t}\lhd \mathbb{X}_n$ is a two sided
ideal, which implies that $\mathscr{N}_{\mathbb{X}_n}(\lambda)$ is a
$\mathbb{X}_n$-invariant subspace.

\begin{proposition}\label{P:cool}
Let $x\neq 0$ and $\lambda\vdash (n-t)\le n$ be a partition, then the following
assertions hold\\
{\rm (a)} $\mathscr{M}_{\mathbb{X}_n}(\lambda)/
\mathscr{N}_{\mathbb{X}_n}(\lambda)$
is irreducible as a $\mathbb{X}_n$-module and $\mathbb{I}_n^{n-t}$-module,
respectively. In particular, $\mathscr{M}_{\mathbb{X}_n}(\lambda)$ is
irreducible if and only if $\mathscr{N}_{\mathbb{X}_n}(\lambda)=0$.\\
{\rm (b)} $\mathscr{N}_{\mathbb{X}_n}(\lambda)$ is a maximal
          $\mathbb{X}_n$-submodule of $\mathscr{M}_{\mathbb{X}_n}(\lambda)$
and $\mathscr{N}_{\mathbb{X}_n}(\lambda)$ is unique.\\
{\rm (c)} For any irreducible $\mathbb{X}_n$-module, $V$, there exists a
partition $\lambda\vdash m\le n$ with the property
$V\cong \mathscr{M}_{\mathbb{X}_n}
(\lambda)/\mathscr{N}_{\mathbb{X}_n}(\lambda)$.\\
\end{proposition}
\begin{proof}
We first prove (a). Since $\mathscr{N}_{\mathbb{X}_n}(\lambda)$ is a
$\mathbb{X}_n$-invariant subspace, $\mathscr{M}_{\mathbb{X}_n}(\lambda)/
\mathscr{N}_{\mathbb{X}_n}(\lambda)$  is a $\mathbb{X}_n$- and 
$\mathbb{I}_n^{n-t}$-module.\\
{\it Claim.} Any $v\in \mathscr{M}_{\mathbb{X}_n}(\lambda)\setminus
\mathscr{N}_{\mathbb{X}_n}(\lambda)$ has
the property $\mathbb{I}_n^{n-t} v=
\mathscr{M}_{\mathbb{X}_n}(\lambda)$.\\
To prove the claim we represent $v=\sum_i \mathfrak{a}_i\otimes v_i$, where
$\mathfrak{a}_i\in \mathscr{I}_n^{n-t}$ and $\text{\rm bot}(\mathfrak{a}_i)=
\text{\rm bot}(\mathfrak{u}_{n,t})$.
Let $\delta_{\mathfrak{b}\mathfrak{a}_i}=1$ if
$\mathfrak{b}\mathfrak{a}_i\neq 0$ and $\delta_{\mathfrak{b}\mathfrak{a}_i}=0$
in $\mathbb{I}_n^{n-t}$, otherwise.
For an arbitrary diagram, $\mathfrak{b}\in\mathscr{I}_n^{n-t}$, we have
\begin{equation}\label{E:qq}
\mathfrak{b}\cdot \sum_i \mathfrak{a}_i\otimes v_i=
\sum_i (\mathfrak{b}\mathfrak{a}_i)\otimes v_i=
\widetilde{\mathfrak{b}}\otimes \sum_i\delta_{\mathfrak{b}
\mathfrak{a}_i}x^{\ell(\mathfrak{a}_i,\mathfrak{b})}
\sigma_{\mathfrak{a}_i,\mathfrak{b}}v_i,
\end{equation}
where
$\ell(\mathfrak{b},\mathfrak{a}_i)$ denotes the number of inner components
in $G(\mathfrak{b},\mathfrak{a}_i)$, 
$\sigma_{\mathfrak{b},\mathfrak{a}_i}\in S_{n-t}$
is such that the diagram $\widetilde{\mathfrak{b}}\in\mathscr{I}^{n-t}_n$ 
has noncrossing verticals, has
$\text{\rm top}(\widetilde{\mathfrak{b}})
=\text{\rm top}(\mathfrak{b})$ and satisfies
\begin{equation}
x^{\ell(\mathfrak{b},\mathfrak{a}_i)}\widetilde{\mathfrak{b}}\,
\sigma_{\mathfrak{b},\mathfrak{a}_i}=\mathfrak{b}\mathfrak{a}_i.
\end{equation}
For any $v\in\mathscr{M}_{\mathbb{X}_n}(\lambda)\setminus
\mathscr{N}_{\mathbb{X}_n}(\lambda)$ we have
$\mathbb{I}_{n}^{n-t}v\neq 0$, whence there exists some
$\mathfrak{b}_0\in\mathscr{I}_n^{n-t}$
such that
\begin{equation}\label{E:b0}
\mathfrak{b}_0\cdot \sum_i \mathfrak{a}_i\otimes v_i=
\sum_i (\mathfrak{b}_0\mathfrak{a}_i)\otimes v_i=
\widetilde{\mathfrak{b}}_0\otimes \sum_i\delta_{\mathfrak{b}_0\mathfrak{a}_i}
x^{\ell(\mathfrak{b}_0,\mathfrak{a}_i)}
\sigma_{\mathfrak{b}_0,\mathfrak{a}_i}v_i\neq 0,
\end{equation}
where $\text{\rm top}(\widetilde{\mathfrak{b}}_0)=
\text{\rm top}(\mathfrak{b}_0)$, $\widetilde{\mathfrak{b}}_0
\in\mathscr{I}_n^{n-t}$ has noncrossing verticals and
$\sigma_{\mathfrak{b}_0,\mathfrak{a}_i}\in S_{n-t}$ is such that
\begin{equation}
x^{\ell(\mathfrak{b}_0,\mathfrak{a}_i)}\widetilde{\mathfrak{b}}_0\,
\sigma_{\mathfrak{b}_0,\mathfrak{a}_i}=\mathfrak{b}_0\mathfrak{a}_i.
\end{equation}
For arbitrary $\mathfrak{b}\in\mathscr{I}_n^{n-t}$ we consider
the element $\mathfrak{b}^\ddagger$ having the properties:
$\text{\rm top}(\mathfrak{b}^\ddagger)=\text{\rm top}(\mathfrak{b})$,
$\text{\rm bot}(\mathfrak{b}^\ddagger)=\text{\rm bot}(\mathfrak{b}_0)$,
having $n-t$ vertical arcs and satisfying
\begin{equation}\label{E:dagger}
\mathfrak{b}^\ddagger\mathfrak{a}_i=
x^{\ell(\mathfrak{b}_0,\mathfrak{a}_i)}\widetilde{\mathfrak{b}^\ddagger}
\sigma_{\mathfrak{b}_0,\mathfrak{a}_i},
\end{equation}
where $\widetilde{\mathfrak{b}^\ddagger}\in\mathscr{I}_n^{n-t}$
has noncrossing verticals and
$\text{\rm top}(\widetilde{\mathfrak{b}^\ddagger})=
\text{\rm top}(\mathfrak{b}^\ddagger)=\text{\rm top}(\mathfrak{b})$.
Multiplying with $\mathfrak{b}^\ddagger$ we obtain
\begin{equation*}
\mathfrak{b}^\ddagger\cdot
\sum_i \mathfrak{a}_i\otimes v_i=
\sum_i (\mathfrak{b}^\ddagger\mathfrak{a}_i)\otimes v_i=
\widetilde{\mathfrak{b}^\ddagger}\otimes \sum_i\delta_{\mathfrak{b}_0
\mathfrak{a}_i}x^{\ell(\mathfrak{b}_0,\mathfrak{a}_i)}
\sigma_{\mathfrak{b}_0,\mathfrak{a}_i}v_i
\neq 0.
\end{equation*}
We set $w=\sum_i\delta_{\mathfrak{b}_0
\mathfrak{a}_i}x^{\ell(\mathfrak{b}_0,\mathfrak{a}_i)}
\sigma_{\mathfrak{b}_0,\mathfrak{a}_i}v_i$ and note that $w\neq 0$ holds.
Since $S^\lambda$ is irreducible, for any $0\neq u$ the elements $\sigma_0 u$,
$\sigma_0\in S_{n-t}$ generate $S^\lambda$.
Since for any $\sigma_0\in S_{n-t}$ there exists some
$g(\sigma_0,\mathfrak{b}^\ddagger)\in \mathbb{I}_n^{n-t}$ with the property
\begin{equation}
g(\sigma_0,\mathfrak{b}^\ddagger)\cdot\widetilde{\mathfrak{b}^\ddagger}=
x^m\widetilde{\mathfrak{b}^\ddagger}\sigma_0 \quad
\text{\rm for some $m\in\mathbb{Z}$,}
\end{equation}
we conclude
\begin{equation}
g(\sigma_0) \cdot \mathfrak{b}^\ddagger\cdot
\sum_i \mathfrak{a}_i\otimes v_i=
g(\sigma_0)\cdot\widetilde{\mathfrak{b}^\ddagger}
\otimes w =x^m
\widetilde{\mathfrak{b}^\ddagger}
\otimes \sigma_0w.
\end{equation}
Accordingly, $\mathbb{I}_n^{n-t}\cdot v = \mathbb{I}_n^{n-t}\mathfrak{u}_{n,t}
\otimes_{S_{n-t}} S^\lambda$ and the Claim is proved.\\
As a result, any nontrivial 
$\mathscr{M}_{\mathbb{X}_n}(\lambda)/\mathscr{N}_{\mathbb{X}_n}
(\lambda)$-element 
generates $\mathscr{M}_{\mathbb{X}_n}(\lambda)/
\mathscr{N}_{\mathbb{X}_n}(\lambda)$, 
which is equivalent to $\mathscr{M}_{\mathbb{X}_n}(\lambda)/
\mathscr{N}_{\mathbb{X}_n}(\lambda)$ being an irreducible
$\mathbb{I}_n^{n-t}$-left module.
This action extends to an unique $\mathbb{X}_n$-left action
with respect to which
$\mathscr{M}_{\mathbb{X}_n}(\lambda)/\mathscr{N}_{\mathbb{X}_n}(\lambda)$ is
an irreducible $\mathbb{X}_n$-module. This proves assertion (a). \\
We next prove (b): the maximality of $\mathscr{N}_{\mathbb{X}_n}(\lambda)$ 
follows
from the irreducibility of $\mathscr{M}_{\mathbb{X}_n}(\lambda)/
\mathscr{N}_{\mathbb{X}_n}(\lambda)$.
It remains to show that $\mathscr{N}_{\mathbb{X}_n}(\lambda)$ is unique.
For this purpose, let $M$ be a maximal $\mathbb{X}_n$-left submodule of
$\mathscr{M}_{\mathbb{X}_n}(\lambda)$ different from 
$\mathscr{N}_{\mathbb{X}_n}(\lambda)$. Then there exist a 
$v\in M\setminus \mathscr{N}_{\mathbb{X}_n}(\lambda)$, which, according
to (a) generates $\mathscr{M}_{\mathbb{X}_n}(\lambda)$. Consequently, any
maximal $\mathscr{M}_{\mathbb{X}_n}(\lambda)$-module, different from
$\mathscr{N}_{\mathbb{X}_n}(\lambda)$, is equal to 
$\mathscr{M}_{\mathbb{X}_n}(\lambda)$ and (b) follows.\\
Next we show (c).
Let $(n-t)$ be the smallest integer with the property $\mathbb{X}_n^{n-t}$
is not acting trivially on $V$.
Consider the set $V_0=\{v\in V\mid \mathbb{I}_n^{n-t}v=0\}$. Clearly, since
$\mathbb{X}_{n}^{n-t}\lhd \mathbb{X}_n$ is a two sided ideal, $V_0$ is an
$\mathbb{X}_n$-invariant subspace and the irreducibility of $V$ implies either
$V_0=0$ or $V_0=V$. By definition of $(n-t)$, there exists a $v\in V$ such that
$\mathbb{I}_n^{n-t}v\neq 0$, whence $V_0=0$.
Therefore, any $0\neq v\in V$ has the property $\mathbb{I}_n^{n-t}v\neq 0$
and $\mathbb{I}_n^{n-t}v$ is $\mathbb{A}_n$-invariant.
Since $V$ is an irreducible $\mathbb{X}_n$-module we have 
$\mathbb{I}_n^{n-t}v=V$.
Accordingly, $V$ is also an irreducible $\mathbb{I}_n^{n-t}$-left module.\\
As an $\mathbb{I}_n^{n-t}$-left module the algebra $\mathbb{I}_n^{n-t}$
decomposes into a direct sum of modules that are isomorphic to
$\mathscr{M}_{\mathbb{X}_n}(\lambda)$, for $\lambda\vdash (n-t)$, i.e.
\begin{equation}\label{E:decompose}
\mathbb{I}_n^{n-t}\cong \bigoplus_{\lambda\vdash (n-t)}n_\lambda\,
                            \mathscr{M}_{\mathbb{X}_n}(\lambda),
\end{equation}
where $n_\lambda$ denotes the multiplicity of
$\mathscr{M}_{\mathbb{X}_n}(\lambda)$ in $\mathbb{I}_n^{n-t}$.
Clearly we have for any $0\neq v\in V$ the surjective morphism of 
$\mathbb{I}_n^{n-t}$-left modules
$\phi_v\colon \mathbb{I}_n^{n-t}\longrightarrow V$, given by 
$\mathfrak{a}\mapsto
\mathfrak{a}\cdot v$.
Accordingly there exists a partition $\lambda\vdash (n-t)$ and a surjective
morphism of $\mathbb{I}_n^{n-t}$-left modules induced by $\phi_v$:
\begin{equation*}
\phi_v^\lambda
\colon \mathscr{M}_{\mathbb{X}_n}(\lambda)\longrightarrow V.
\end{equation*}
Assertion (a) and (b) imply $\text{\rm ker}(\phi_v^\lambda)=
\mathscr{N}_{\mathbb{X}_n}(\lambda)$, i.e.~we have
$\mathscr{M}_{\mathbb{X}_n}(\lambda)/\mathscr{N}_{\mathbb{X}_n}(\lambda)
\cong V$
and the proof of Proposition~\ref{P:cool} is complete.
\end{proof}

The next result connects semisimplicity of $\mathbb{X}_n$ with the existence
of nontrivial morphisms between the modules 
$\mathscr{M}_{\mathbb{X}_n}(\lambda)$
and $\mathscr{M}_{\mathbb{X}_n}(\mu)$.
Indeed, if $\mathbb{X}_n$ is not semisimple, then there exists some module
$\mathscr{M}_{\mathbb{X}_n}(\mu)$, $\mu\vdash m<n$ with a nontrivial
maximal submodule $\mathscr{N}_{\mathbb{X}_n}(\mu)$.
In the following we denote by $\text{\rm Rad}(\mathbb{X}_n)$ the Jacobson 
radical
of $\mathbb{X}_n$, i.e.~$\mathbb{X}_n$ is semisimple if and only if
$\text{\rm Rad}(\mathbb{X}_n)=0$.

\begin{proposition}\label{P:cool2}
If $\mathbb{X}_n$ is not semisimple, then there exist two partitions
$\mu,\lambda$, where $\vert \mu\vert<\vert\lambda\vert \le n$ and a
short exact sequence of $\mathbb{X}_n$-modules
\begin{equation}
\diagram
0\rto & \mathscr{N}_{\mathbb{X}_n}(\lambda) \rto &
\mathscr{M}_{\mathbb{X}_n}(\lambda) \rto^{\varphi_n} &
 \mathscr{M}_{\mathbb{X}_n}(\mu). \\
\enddiagram
\end{equation}
\end{proposition}
\begin{proof}
Suppose first that $\mathscr{M}_{\mathbb{X}_n}(\mu)$ is for any partition
$\mu\vdash m$,
$m<n$, irreducible. We claim that $\mathbb{X}_n$ is in this case semisimple.
To this end we observe that for $\mu\vdash n$, we have
$\mathscr{M}_{\mathbb{X}_n}(\mu)\cong S^\mu$, i.e.~for arbitrary partition
$\mu$, the module $\mathscr{M}_{\mathbb{X}_n}(\mu)$ is irreducible.
In view of
\begin{equation*}
\mathbb{I}_n^m\cong 
\bigoplus_{\mu\vdash m}n_\mu\,\mathscr{M}_{\mathbb{X}_n}(\mu),
\end{equation*}
for any $0\le m\le n$, the $F$-algebras
$\mathbb{X}_n^m/\mathbb{X}_n^{m-1}\cong \mathbb{I}_n^m$ and in particular
$\mathbb{X}_n^0\cong \mathbb{I}_n^0$, are semisimple.
Since $\text{\rm Rad}(\mathbb{X}_n^m)$ is a nilpotent ideal so is
$(\text{\rm Rad}(\mathbb{X}_n^m)+\mathbb{X}_n^{m-1})/\mathbb{X}_n^{m-1}$
and we obtain
\begin{equation*}
(\text{\rm Rad}(\mathbb{X}_n^m)+\mathbb{X}_n^{m-1})/\mathbb{X}_n^{m-1}\subset
\text{\rm Rad}(\mathbb{X}_n^m/\mathbb{X}_n^{m-1})=0.
\end{equation*}
We next observe using
$\text{\rm Rad}(\mathbb{X}_n^m)\cap \mathbb{X}_n^{m-1}=
\text{\rm Rad}(\mathbb{X}_n^{m-1})$
\begin{eqnarray*}
(\text{\rm Rad}(\mathbb{X}_n^m)+\mathbb{X}_n^{m-1})/\mathbb{X}_n^{m-1}
& \cong & \text{\rm Rad}(\mathbb{X}_n^m)/(\text{\rm Rad}(\mathbb{X}_n^m)\cap
\mathbb{X}_n^{m-1})\\
&\cong & \text{\rm Rad}(\mathbb{X}_n^m)/\text{\rm Rad}
(\mathbb{X}_n^{m-1}).
\end{eqnarray*}
Consequently we have for $1\le m\le n$ the inclusion $\text{\rm Rad}(
\mathbb{X}_n^m)\subset\text{\rm Rad}(\mathbb{X}_n^{m-1})$, which implies
$\text{\rm Rad}(\mathbb{X}_n)\subset \text{\rm Rad}(
\mathbb{X}^{0}_n)=0$, i.e.~$\mathbb{X}_n$ is semisimple.\\
Thus, if $\mathbb{X}_n$
is not semisimple, there exists a partition $\mu\vdash m$, $m<n$, such that
$\mathscr{M}_{\mathbb{X}_n}(\mu)$ is not irreducible. Then there exists
according to Proposition~\ref{P:cool}, assertion (b), the nontrivial,
maximal submodule $\mathscr{N}_{\mathbb{X}_n}(\mu)\subset
\mathscr{M}_{\mathbb{X}_n}(\mu)$.
Let $m_0$ be the smallest integer such that $\mathbb{X}_n^{m_0}$ acts 
nontrivially
on $\mathscr{N}_{\mathbb{X}_n}(\mu)$. By definition we have for any $v\in
\mathscr{N}_{\mathbb{X}_n}(\mu)$,
$\mathbb{I}_n^mv=0$, whence $m<m_0$. $\mathscr{N}_{\mathbb{X}_n}(\mu)$ is 
then a nontrivial
$\mathbb{I}_n^{m_0}$-left module and there exists an irreducible
$\mathbb{I}_n^{m_0}$-submodule $W\subset\mathscr{N}_{\mathbb{X}_n}(\mu)$.
According to Proposition~\ref{P:cool}, assertion (c), $W$ is isomorphic to
$\mathscr{M}_{\mathbb{X}_n}(\lambda)/\mathscr{N}_{\mathbb{X}_n}(\lambda)$ for
some $\lambda\vdash m_0$,
i.e.~$\vert \mu\vert <\vert \lambda\vert \le n$.
Therefore there exists a partition $\lambda$ and a nontrivial morphism of
$\mathbb{X}_n$-modules $\varphi_n\colon \mathscr{M}_{\mathbb{X}_n}(\lambda)
\longrightarrow
\mathscr{M}_{\mathbb{X}_n}(\mu)$, such that $\text{\rm ker}(\varphi_n)=
\mathscr{N}_{\mathbb{X}_n}(\lambda)$
and $\vert \mu\vert <\vert \lambda\vert \le n$ and
the proposition follows.
\end{proof}


\section{Restriction and induction}\label{S:res}


We shall begin by showing that $\mathbb{A}_n$ has the generators
$\mathfrak{g}_{i-1},\mathfrak{e}_{i-1},\mathfrak{u}_{j}$, $2\le i\le
n$, $1\le j\le n$.

\begin{lemma}\label{L:rep}
Any diagram $\mathfrak{a}\in\mathscr{A}_{n}$ is either contained in
$\mathscr{A}_{n-1}$ or of the form
\begin{equation}
\mathfrak{a} =
\mathfrak{a}'\, \mathfrak{x}\,\mathfrak{b}',
\quad
\mathfrak{a}',\mathfrak{b}'\in\mathscr{A}_{n-1},\
\mathfrak{x}\in\{\mathfrak{g}_{n-1},\mathfrak{e}_{n-1},\mathfrak{u}_{n}\} \ .
\end{equation}
In particular, we have $\mathbb{A}_n=\langle S_n,\mathfrak{e}_{n-1},
\mathfrak{u}_{n}\rangle$ and $\mathbb{L}_n=\langle S_n,\mathfrak{u}_n\rangle$.
\end{lemma}
\begin{proof}
Any diagram not contained in $\mathscr{A}_{n-1}$ has either (a) none or
two loops at the vertices $n,n'$, (b) exactly one loop over $n$($n'$)
and at least one loop over some vertex $i'$($i$), where $i<n$ or
(c) exactly one loop over $n$($n'$) and no loops
over $i'$($i$), where $i<n$. From this we derive
\begin{equation}\label{E:OP}
\mathfrak{a} = \mathfrak{a}'\, \mathfrak{y}\,\mathfrak{b}' \quad
\text{\rm where}\quad \mathfrak{a}',\mathfrak{b}'\in\mathscr{A}_{n-1}\
\text{\rm and}\
\mathfrak{y}\in
\begin{cases}
\{\mathfrak{g}_{n-1},\mathfrak{e}_{n-1},\mathfrak{u}_{n}\}
& \text{\rm  {\rm (a)}} \\
\{\mathfrak{d}_1,\mathfrak{d}_1^*,
\mathfrak{d}_2,\mathfrak{d}_2^*\}
& \text{\rm   {\rm (b)}} \\
\{\mathfrak{d}_3,\mathfrak{d}_3^*,
\mathfrak{d}_4,\mathfrak{d}_4^*\}
& \text{\rm  {\rm (c)}}
\end{cases}
\end{equation}
where
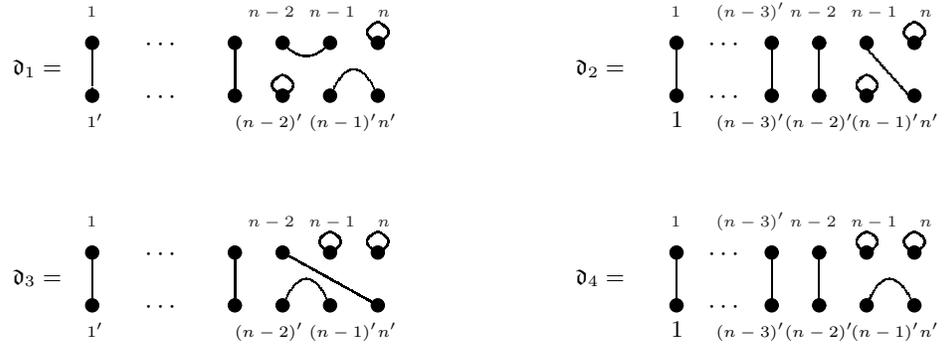
\begin{center}
\begin{picture}(280,20)
\put(-2,30){\tiny{$1$}}
\put(59,30){\tiny{$n-2$}}\put(82,30){\tiny{$n-1$}}
\put(108,30){\tiny{$n$}} \put(219,30){\tiny{$1$}}
\put(236,30){\tiny{$(n-3)'$}} \put(264,30){\tiny{$n-2$}}
\put(287,30){\tiny{$n-1$}}\put(313,30){\tiny{$n$}}

\put(0,0){\circle*{5}} 
\put(54,0){\circle*{5}} \put(72,0){\circle*{5}}
\put(90,0){\circle*{5}}\put(108,0){\circle*{5}}
\put(221,0){\circle*{5}}
\put(257,0){\circle*{5}}\put(275,0){\circle*{5}}
\put(293,0){\circle*{5}}\put(311,0){\circle*{5}}
\put(0,20){\circle*{5}} 
\put(54,20){\circle*{5}} \put(72,20){\circle*{5}}
\put(90,20){\circle*{5}} \put(108,20){\circle*{5}}
\put(221,20){\circle*{5}}
\put(257,20){\circle*{5}}\put(275,20){\circle*{5}}
\put(293,20){\circle*{5}}\put(311,20){\circle*{5}}
\put(-2,-12){\tiny{$1'$}}
\put(54,-12){\tiny{$(n-2)'$}}\put(82,-12){\tiny{$(n-1)'$}}
\put(108,-12){\tiny{$n'$}} \put(219,-12){\small{$1$}}
\put(236,-12){\tiny{$(n-3)'$}} \put(262,-12){\tiny{$(n-2)'$}}
\put(287,-12){\tiny{$(n-1)'$}}\put(313,-12){\tiny{$n'$}}
\put(-30,8){\small{$\mathfrak{d}_1=$}}
\put(183,8){\small{$\mathfrak{d}_2=$}}
\put(20,-3){$\cdots$}\put(20,17){$\cdots$}
\put(233,-3){$\cdots$}\put(233,17){$\cdots$}
\qbezier[40](0,20)(0,20)(0,0) \qbezier[40](54,20)(54,10)(54,0)
\qbezier[40](72,20)(81,10)(90,20)\qbezier[40](90,0)(99,20)(108,0)
\qbezier[40](221,20)(221,10)(221,0)
\qbezier[40](257,20)(257,10)(257,0)
\qbezier[40](275,20)(275,10)(275,0)
\qbezier[40](291,20)(300,10)(309,0) \qbezier[50](72,0)(64,4)(72,8)
\qbezier[50](72,0)(80,4)(72,8) \qbezier[50](108,20)(100,24)(108,28)
\qbezier[50](108,20)(116,24)(108,28)\qbezier[50](293,0)(285,4)(293,8)
\qbezier[50](293,0)(301,4)(293,8)\qbezier[50](311,20)(303,24)(311,28)
\qbezier[50](311,20)(319,24)(311,28)
\end{picture}
\end{center}
\begin{center}
\begin{picture}(280,50)
\put(-2,30){\tiny{$1$}}
\put(59,30){\tiny{$n-2$}}\put(82,30){\tiny{$n-1$}}
\put(108,30){\tiny{$n$}} \put(219,30){\tiny{$1$}}
\put(236,30){\tiny{$(n-3)'$}} \put(264,30){\tiny{$n-2$}}
\put(287,30){\tiny{$n-1$}}\put(313,30){\tiny{$n$}}

\put(0,0){\circle*{5}} 
\put(54,0){\circle*{5}} \put(72,0){\circle*{5}}
\put(90,0){\circle*{5}}\put(108,0){\circle*{5}}
\put(221,0){\circle*{5}}
\put(257,0){\circle*{5}}\put(275,0){\circle*{5}}
\put(293,0){\circle*{5}}\put(311,0){\circle*{5}}
\put(0,20){\circle*{5}} 
\put(54,20){\circle*{5}} \put(72,20){\circle*{5}}
\put(90,20){\circle*{5}} \put(108,20){\circle*{5}}
\put(221,20){\circle*{5}}
\put(257,20){\circle*{5}}\put(275,20){\circle*{5}}
\put(293,20){\circle*{5}}\put(311,20){\circle*{5}}
\put(-2,-12){\tiny{$1'$}}
\put(54,-12){\tiny{$(n-2)'$}}\put(82,-12){\tiny{$(n-1)'$}}
\put(108,-12){\tiny{$n'$}} \put(219,-12){\small{$1$}}
\put(236,-12){\tiny{$(n-3)'$}} \put(262,-12){\tiny{$(n-2)'$}}
\put(287,-12){\tiny{$(n-1)'$}}\put(313,-12){\tiny{$n'$}}
\put(-30,8){\small{$\mathfrak{d}_3=$}}
\put(183,8){\small{$\mathfrak{d}_4=$}}
\put(20,-3){$\cdots$}\put(20,17){$\cdots$}
\put(233,-3){$\cdots$}\put(233,17){$\cdots$}
\qbezier[40](0,20)(0,20)(0,0) \qbezier[40](54,20)(54,10)(54,0)
\qbezier[140](72,20)(90,10)(108,0) \qbezier[40](72,0)(81,20)(90,0)
\qbezier[40](221,20)(221,10)(221,0)
\qbezier[40](257,20)(257,10)(257,0)
\qbezier[40](275,20)(275,10)(275,0)
\qbezier[40](293,0)(301,20)(311,0) \qbezier[40](90,20)(82,24)(90,28)
\qbezier[40](90,20)(98,24)(90,28)\qbezier[40](293,20)(285,24)(293,28)
\qbezier[40](293,20)(301,24)(293,28)\qbezier[40](108,20)(100,24)(108,28)
\qbezier[40](108,20)(116,24)(108,28)\qbezier[40](311,20)(303,24)(311,28)
\qbezier[40](311,20)(319,24)(311,28)
\end{picture}
\end{center}
We can express the diagrams $\mathfrak{d}_1,\dots,\mathfrak{d}_4$ via
the generators $\mathfrak{g}_i$, $\mathfrak{e}_i$ and $\mathfrak{u}_i$
as follows
\begin{eqnarray*}
\mathfrak{d}_1  =  \mathfrak{e}_{n-2}\,\mathfrak{u}_{n-1}\,
                     \mathfrak{u}_{n-2}\, \mathfrak{e}_{n-1}, \quad
\mathfrak{d}_2  =  \mathfrak{u}_n\,\mathfrak{g}_{n-1}, \quad
\mathfrak{d}_3  =  \mathfrak{u}_n\,\mathfrak{e}_{n-1}\,\mathfrak{e}_{n-2},
\quad \mathfrak{d}_4  =  \mathfrak{u}_{n-1}\, \mathfrak{e}_{n-1}.
\end{eqnarray*}
We next observe that the relations
\begin{eqnarray*}\label{E:basic1}
\,\,\,\,\,\,\,\,\,\,\mathfrak{u}_n\,\mathfrak{e}_{n-1}=
\mathfrak{u}_{n-1}\,\mathfrak{e}_{n-1}\,\,\,\,\,\,\,\,\,\,
\,\,\,\,\,\,\,\,\,\,\,\,\, &
\mathfrak{g}_{n-1}\,\mathfrak{u}_{n-1}=\mathfrak{u}_{n}\,
\mathfrak{g}_{n-1}\,\,\,\,\,\,\,\,\,\,\,\,\,\,\,\,\,\,\,\,\,\,\,\,\,\,\,\,
&
\mathfrak{u}_{n-1}\,\mathfrak{g}_{n-1}=\mathfrak{g}_{n-1}\,\mathfrak{u}_{n}
\end{eqnarray*}
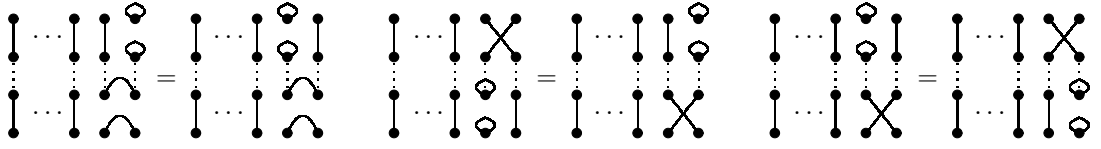
\begin{center}
\setlength{\unitlength}{0,72pt}
\begin{picture}(560,60)
\put(0,0){\circle*{5}}\put(32,0){\circle*{5}}\put(48,0){\circle*{5}}
\put(64,0){\circle*{5}}\put(96,0){\circle*{5}}\put(128,0){\circle*{5}}
\put(144,0){\circle*{5}}\put(160,0){\circle*{5}}
\put(200,0){\circle*{5}}\put(232,0){\circle*{5}}\put(248,0){\circle*{5}}
\put(264,0){\circle*{5}}\put(296,0){\circle*{5}}\put(328,0){\circle*{5}}
\put(344,0){\circle*{5}}\put(360,0){\circle*{5}}
\put(400,0){\circle*{5}}\put(432,0){\circle*{5}}\put(448,0){\circle*{5}}
\put(464,0){\circle*{5}}\put(496,0){\circle*{5}}\put(528,0){\circle*{5}}
\put(544,0){\circle*{5}}\put(560,0){\circle*{5}}
\put(0,20){\circle*{5}}\put(32,20){\circle*{5}}\put(48,20){\circle*{5}}
\put(64,20){\circle*{5}}\put(96,20){\circle*{5}}\put(128,20){\circle*{5}}
\put(144,20){\circle*{5}}\put(160,20){\circle*{5}}
\put(200,20){\circle*{5}}\put(232,20){\circle*{5}}\put(248,20){\circle*{5}}
\put(264,20){\circle*{5}}\put(296,20){\circle*{5}}\put(328,20){\circle*{5}}
\put(344,20){\circle*{5}}\put(360,20){\circle*{5}}
\put(400,20){\circle*{5}}\put(432,20){\circle*{5}}\put(448,20){\circle*{5}}
\put(464,20){\circle*{5}}\put(496,20){\circle*{5}}\put(528,20){\circle*{5}}
\put(544,20){\circle*{5}}\put(560,20){\circle*{5}}
\put(0,40){\circle*{5}}\put(32,40){\circle*{5}}\put(48,40){\circle*{5}}
\put(64,40){\circle*{5}}\put(96,40){\circle*{5}}\put(128,40){\circle*{5}}
\put(144,40){\circle*{5}}\put(160,40){\circle*{5}}
\put(200,40){\circle*{5}}\put(232,40){\circle*{5}}\put(248,40){\circle*{5}}
\put(264,40){\circle*{5}}\put(296,40){\circle*{5}}\put(328,40){\circle*{5}}
\put(344,40){\circle*{5}}\put(360,40){\circle*{5}}
\put(400,40){\circle*{5}}\put(432,40){\circle*{5}}\put(448,40){\circle*{5}}
\put(464,40){\circle*{5}}\put(496,40){\circle*{5}}\put(528,40){\circle*{5}}
\put(544,40){\circle*{5}}\put(560,40){\circle*{5}}
\put(0,60){\circle*{5}}\put(32,60){\circle*{5}}\put(48,60){\circle*{5}}
\put(64,60){\circle*{5}}\put(96,60){\circle*{5}}\put(128,60){\circle*{5}}
\put(144,60){\circle*{5}}\put(160,60){\circle*{5}}
\put(200,60){\circle*{5}}\put(232,60){\circle*{5}}\put(248,60){\circle*{5}}
\put(264,60){\circle*{5}}\put(296,60){\circle*{5}}\put(328,60){\circle*{5}}
\put(344,60){\circle*{5}}\put(360,60){\circle*{5}}
\put(400,60){\circle*{5}}\put(432,60){\circle*{5}}\put(448,60){\circle*{5}}
\put(464,60){\circle*{5}}\put(496,60){\circle*{5}}\put(528,60){\circle*{5}}
\put(544,60){\circle*{5}}\put(560,60){\circle*{5}}
\qbezier[5](0,20)(0,30)(0,40)\qbezier[5](32,20)(32,30)(32,40)
\qbezier[5](48,20)(48,30)(48,40)\qbezier[5](64,20)(64,30)(64,40)
\qbezier[5](96,20)(96,30)(96,40)\qbezier[5](128,20)(128,30)(128,40)
\qbezier[5](144,20)(144,30)(144,40)\qbezier[5](160,20)(160,30)(160,40)
\qbezier[5](200,20)(200,30)(200,40)\qbezier[5](232,20)(232,30)(232,40)
\qbezier[5](248,20)(248,30)(248,40)\qbezier[5](264,20)(264,30)(264,40)
\qbezier[5](296,20)(296,30)(296,40)\qbezier[5](328,20)(328,30)(328,40)
\qbezier[5](344,20)(344,30)(344,40)\qbezier[5](360,20)(360,30)(360,40)
\qbezier[5](400,20)(400,30)(400,40)\qbezier[5](432,20)(432,30)(432,40)
\qbezier[5](448,20)(448,30)(448,40)\qbezier[5](464,20)(464,30)(464,40)
\qbezier[5](496,20)(496,30)(496,40)\qbezier[5](528,20)(528,30)(528,40)
\qbezier[5](544,20)(544,30)(544,40)\qbezier[5](560,20)(560,30)(560,40)
\qbezier[100](0,0)(0,10)(0,20)\qbezier[100](0,40)(0,50)(0,60)
\qbezier[100](32,0)(32,10)(32,20)\qbezier[100](32,40)(32,50)(32,60)
\qbezier[100](48,0)(56,18)(64,0)\qbezier[100](48,20)(56,38)(64,20)
\qbezier[100](48,40)(48,50)(48,60)\qbezier[100](64,40)(54,44)(64,48)
\qbezier[100](64,40)(74,44)(64,48)\qbezier[100](64,60)(54,64)(64,68)
\qbezier[100](64,60)(74,64)(64,68)\qbezier[100](96,0)(96,10)(96,20)
\qbezier[100](96,40)(96,50)(96,60)\qbezier[100](128,40)(128,50)(128,60)
\qbezier[100](128,0)(128,10)(128,20)\qbezier[100](144,0)(152,18)(160,0)
\qbezier[100](144,20)(152,38)(160,20)\qbezier[100](144,40)(134,44)(144,48)
\qbezier[100](144,40)(154,44)(144,48)\qbezier[100](144,60)(134,64)(144,68)
\qbezier[100](144,60)(154,64)(144,68)\qbezier[100](160,40)(160,50)(160,60)
\qbezier[100](200,0)(200,10)(200,20)\qbezier[100](200,40)(200,50)(200,60)
\qbezier[100](232,0)(232,10)(232,20)\qbezier[100](232,40)(232,50)(232,60)
\qbezier[100](248,60)(256,50)(264,40)\qbezier[100](264,60)(256,50)(248,40)
\qbezier[100](248,0)(238,4)(248,8)\qbezier[100](248,0)(258,4)(248,8)
\qbezier[100](248,20)(238,24)(248,28)\qbezier[100](248,20)(258,24)(248,28)
\qbezier[100](264,0)(264,10)(264,20)
\qbezier[100](296,0)(296,10)(296,20)\qbezier[100](296,40)(296,50)(296,60)
\qbezier[100](328,0)(328,10)(328,20)\qbezier[100](328,40)(328,50)(328,60)
\qbezier[100](344,0)(352,10)(360,20)\qbezier[100](344,20)(352,10)(360,0)
\qbezier[100](344,40)(344,50)(344,60)\qbezier[100](360,40)(350,44)(360,48)
\qbezier[100](360,40)(370,44)(360,48)
\qbezier[100](360,60)(350,64)(360,68)
\qbezier[100](360,60)(370,64)(360,68)
\qbezier[100](400,0)(400,10)(400,20)\qbezier[100](400,40)(400,50)(400,60)
\qbezier[100](432,0)(432,10)(432,20)\qbezier[100](432,40)(432,50)(432,60)
\qbezier[100](448,40)(438,44)(448,48)\qbezier[100](448,40)(458,44)(448,48)
\qbezier[100](448,60)(438,64)(448,68)\qbezier[100](448,60)(458,64)(448,68)
\qbezier[100](464,40)(464,50)(464,60)\qbezier[100](448,0)(456,10)(464,20)
\qbezier[100](448,20)(456,10)(464,0)\qbezier[100](496,0)(496,10)(496,20)
\qbezier[100](496,40)(496,50)(496,60)\qbezier[100](528,0)(528,10)(528,20)
\qbezier[100](528,40)(528,50)(528,60)\qbezier[100](544,60)(552,50)(560,40)
\qbezier[100](544,40)(552,50)(560,60)\qbezier[100](544,0)(544,10)(544,20)
\qbezier[100](560,0)(550,4)(560,8)\qbezier[100](560,0)(570,4)(560,8)
\qbezier[100](560,20)(550,24)(560,28)\qbezier[100](560,20)(570,24)(560,28)
\put(75,25){$=$}\put(275,25){$=$}\put(475,25){$=$}
\put(10,10){$\dots$}\put(10,50){$\dots$}\put(210,10){$\dots$}
\put(210,50){$\dots$} \put(105,10){$\dots$}\put(105,50){$\dots$}
\put(305,10){$\dots$}\put(305,50){$\dots$}
\put(505,10){$\dots$}\put(505,50){$\dots$}
\put(410,10){$\dots$}\put(410,50){$\dots$}
\end{picture}
\end{center}
imply eq.~(\ref{E:OP}) from which the lemma follows.
\end{proof}

The next theorem analyzes the restriction in $\mathbb{A}_n$ and
follows the ideas of Doran {\it et al.} \cite{Doran} in the case of 
$\mathbb{B}_n$. We find the following new phenomenon for $\mathbb{A}_n$: 
for $\lambda\vdash (n-t)$, where $t\ge 1$, there exists an embedding of 
$\mathscr{M}_{\mathbb{A}_{n-1}}(\lambda)$ into 
$\text{\rm res}_{n-1}(\mathscr{M}_{\mathbb{A}_n}(\lambda))$.
Such an embedding does not exist for $\mathbb{B}_n$. We shall employ
it in Lemma~\ref{L:one} in order to show that if $\text{\rm
hom}_{\mathbb{X}_n}(\mathscr{M}_{\mathbb{X}_n}(\lambda),
\mathscr{M}_{\mathbb{X}_n}(\mu))\neq 0$ then we can, without loss of
generality, assume that $\lambda\vdash n$.
\begin{theorem}\label{T:realdeal}
Let $n,t\in\mathbb{N}$ and $\lambda\vdash (n-t)$ where $1\le  t \le n$. 
Then there exists the exact sequence of
$\mathbb{A}_{n-1}$-modules
\begin{eqnarray}\label{E:res}
0\longrightarrow \bigoplus_{\alpha\sqsubseteq \lambda }
\mathscr{M}_{\mathbb{A}_{n-1}}(\alpha)
\longrightarrow \text{\rm res}_{n-1}(\mathscr{M}_{\mathbb{A}_n}(\lambda))
\longrightarrow \bigoplus_{\lambda\sqsubset \beta}
\mathscr{M}_{\mathbb{A}_{n-1}}(\beta)
\longrightarrow 0.
\end{eqnarray}
\end{theorem}
\begin{proof}
{\it Claim 1}. There exists the following short exact sequence of
$\mathbb{A}_{n-1}$-left modules
\begin{equation}
0\longrightarrow \bigoplus_{\alpha\sqsubseteq \lambda}
\mathscr{M}_{\mathbb{A}_{n-1}}
(\alpha)\longrightarrow \text{\rm res}_{n-1}
(\mathscr{M}_{\mathbb{A}_n}(\lambda)).
\end{equation}
Let $F_{n}^1(\lambda)$ denote the
$\mathscr{M}_{\mathbb{A}_n}(\lambda)$-subspace generated by all
tensors $\mathfrak{a}\otimes w$,
where $\mathfrak{a}$ is a $\mathbb{I}_{n}^{n-t}\mathfrak{u}_{n,t}$-diagram
in which all vertical edges are noncrossing and the top-vertex $n$ is
incident to a vertical edge.
Obviously, any tensor $\mathfrak{b}\otimes w\in\mathbb{I}_{n}^{n-t}
\mathfrak{u}_{n,t}\otimes_{S_{n-t}}S^\lambda$ in which $n$ is
incident to a vertical edge, satisfies $\mathfrak{b}\otimes w=
\mathfrak{a}\otimes\sigma w$ for some $\sigma\in S_{n-t}$.
Let $f_1(\mathfrak{a})$ be the diagram derived from $\mathfrak{a}$
by removing $n$ and $(n-t)'$ and by shifting all bottom vertices
$\ell'>(n-t)'$ down by one. $f_1$ induces the mapping
\begin{equation}\label{E:1-}
\begin{split}
\varphi_1\colon F_{n}^1(\lambda) & \longrightarrow
\mathbb{I}_{n-1}^{n-1-t}\mathfrak{u}_{n-1,t}\otimes_{S_{(n-1)-t}}
\text{\rm res}_{S_{n-1-t}}(S^\lambda) \\
\mathfrak{a}\otimes w & \longmapsto f_1(\mathfrak{a})\otimes w.
\qquad\qquad \qquad
\end{split}
\end{equation}
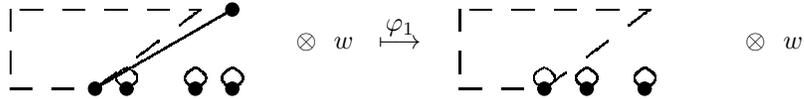
\begin{center}
\begin{picture}(320,15)
\put(70,0){\circle*{5}} \put(84,0){\circle*{5}}
\put(202,0){\circle*{5}}
\put(84,30){\circle*{5}}\put(44,0){\circle*{5}}\put(32,0){\circle*{5}}
\put(240,0){\circle*{5}} 
\put(218,0){\circle*{5}}
\qbezier[100](0,0)(4,0)(8,0)\qbezier[100](16,0)(20,0)(24,0)
\qbezier[50](0,6)(0,8)(0,14)\qbezier[100](0,22)(0,26)(0,30)
\qbezier[20](20,30)(24,30)(28,30) \qbezier[20](52,30)(56,30)(60,30)
\qbezier[20](68,30)(70,30)(72,30)\qbezier[20](32,0)(36,3)(40,6)
\qbezier[20](48,12)(52,15)(56,18)
\qbezier[20](64,24)(68,27)(72,30)\qbezier[20](4,30)(8,30)(12,30)
\qbezier[20](36,30)(40,30)(44,30)\qbezier(32,0)(58,15)(84,30)
\qbezier[20](44,0)(36,4)(44,8)\qbezier[20](44,0)(52,4)(44,8)
\qbezier[100](170,0)(174,0)(178,0)\qbezier[100](186,0)(190,0)(194,0)
\qbezier[50](170,6)(170,8)(170,14)\qbezier[100](170,22)(170,26)(170,30)
\qbezier[20](190,30)(194,30)(198,30)
\qbezier[20](222,30)(226,30)(230,30)
\qbezier[20](238,30)(240,30)(242,30)\qbezier[20](202,0)(206,3)(210,6)
\qbezier[20](218,12)(222,15)(226,18)
\qbezier[20](234,24)(238,27)(242,30)\qbezier[20](174,30)(178,30)(182,30)
\qbezier[20](206,30)(210,30)(214,30)
\qbezier(70,0)(62,4)(70,8) \qbezier(70,0)(78,4)(70,8)
\qbezier(84,0)(76,4)(84,8) \qbezier(84,0)(92,4)(84,8)
 \qbezier(240,0)(232,4)(240,8)
\qbezier(240,0)(248,4)(240,8)
\qbezier[20](202,0)(194,4)(202,8)\qbezier[20](202,0)(210,4)(202,8)
\qbezier(218,0)(210,4)(218,8) \qbezier(218,0)(226,4)(218,8)
\put(108,15){$\otimes\ \ w\ \ \longmapsto
$}\put(142,22){$\varphi_{1}$} \put(278,15){$\otimes\ \ w$}
\end{picture}
\end{center}
We next prove that $\varphi_1$ is bijective.
Indeed, for any $\mathbb{I}_{n-1}^{n-1-t}\mathfrak{u}_{n-1,t}
$-diagram, $\mathfrak{x}$, there exists a unique permutation
$\sigma_0\in S_{n-1-t}$ such that the vertical edges in
$\mathfrak{x}\sigma_0$ are noncrossing. Furthermore
we have $\mathfrak{x}\otimes w=\mathfrak{x}\sigma_0 \otimes \sigma_0^{-1}w$.
Clearly, the tensor $\mathfrak{x}\sigma_0\otimes \sigma_0^{-1}w$ has a unique
$\varphi_1$-preimage, $f_1^{-1}(\mathfrak{x}\sigma_0)\otimes \sigma_0^{-1}w$
where $f_1^{-1}(\mathfrak{x}\sigma_0)$ is obtained by shifting the bottom
vertices $\ell' \ge (n-t)'$ up by one and by adding the vertices $n$ and
$(n-t)'$ together with an vertical edge connecting them.
This proves that $\varphi_1$ is bijective.\\
We next show that $F_n^1(\lambda)$ is, via the natural embedding
$\epsilon_n\colon\mathbb{A}_{n-1}\longrightarrow \mathbb{A}_n$, 
an $\mathbb{A}_{n-1}$-module.
In view of Lemma~\ref{L:rep} it suffices to show
\begin{equation*}
\mathfrak{x}\cdot (\mathfrak{a}\otimes v_i)\in F_n^1(\lambda),
\end{equation*}
where $\mathfrak{x}\in \{\sigma,\mathfrak{e}_{i},\mathfrak{u}_{j}\}$,
$1\le j\le n-1$, $1\le i\le n-2$ and $\sigma\in S_{n-1}$.
Let $\mathfrak{a}$ be a $\mathbb{I}_{n}^{n-t}\mathfrak{u}_{n,t}$-diagram
in which all vertical edges are noncrossing and the top-vertex $n$ is
incident to a vertical edge and let $\sigma\in S_{n-1}$. Then there exist
a unique $\mathbb{I}_{n}^{n-t}\mathfrak{u}_{n,t}$-diagram, $\mathfrak{a}'$,
with noncrossing vertical edges, in which $n$ is connected to $(n-t)'$ and
a permutation $\sigma_0 \in S_{(n-1)-t}$ such that $\sigma \mathfrak{a}
= \mathfrak{a'}\sigma_0$ holds. Consequently,
$$
\sigma\cdot (\mathfrak{a}\otimes w)=\mathfrak{a}'\sigma_0\otimes w=
\mathfrak{a}'\otimes\sigma_0w,
$$
i.e.~$\sigma\cdot (\mathfrak{a}\otimes w)\in F_n^1(\lambda)$.
The cases $\mathfrak{e}_{i}\cdot \mathfrak{a}\otimes v_j$
and $\mathfrak{u}_{i+1}\cdot \mathfrak{a}\otimes v_j$ follow analogously.
We next show that $\varphi_1$ is an isomorphism of $\mathbb{A}_{n-1}$-modules,
that is we prove
$\mathfrak{b}\cdot \varphi_1(\zeta)=\varphi_1(\mathfrak{b}\cdot
\zeta)$. Indeed, for
$\mathfrak{x}\in\{\sigma,\mathfrak{e}_i,\mathfrak{u}_{j}\}$
$$
\mathfrak{x}\cdot (f(\mathfrak{a})\otimes w)=
f(\mathfrak{x}\,\mathfrak{a})\otimes w,
$$
since neither vertex $n$ or its incident bottom vertex $(n-t)'$ are
affected by left multiplication with the elements $\sigma,\mathfrak{e}_i,
\mathfrak{u}_{j}$.\\
Let $F_{n}^2(\lambda)\subset \mathscr{M}_{\mathbb{A}_n}(\lambda)$ be the 
subspace generated by all tensors $\mathfrak{a}\otimes v_i$, where
$\mathfrak{a}\in \mathbb{I}_{n}^{n-t}\mathfrak{u}_{n,t}$ is a diagram having a
loop at vertex $n$. Let
$f_2(\mathfrak{a})\in\mathbb{I}_{n-1}^{n-t}\mathfrak{u}_{n-1,t-1}$ be the
diagram obtained by removing the vertices $n$ and $n'$ together with their
loops. It is straightforward to show that $f_2$ induces the isomorphism of
$\mathbb{A}_{n-1}$-modules
\begin{equation}\label{E:1+}
\begin{split}
\varphi_2\colon F^2_{n}(\lambda) & \longrightarrow
{\mathbb{I}_{n-1}^{n-t}\mathfrak{u}_{n-1,t-1}\otimes_{S_{n-t}}
S^\lambda} \\
\mathfrak{a}\otimes w & \longmapsto  f_2(\mathfrak{a})\otimes w,
\end{split}
\end{equation}
where $\mathbb{I}_{n-1}^{n-t}\mathfrak{u}_{n-1,t-1}\otimes_{S_{n-t}}
S^\lambda\cong \mathscr{M}_{\mathbb{A}_{n-1}}(\lambda)$.
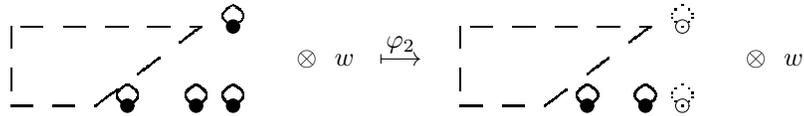
\begin{center}
\begin{picture}(290,20)
\put(70,0){\circle*{5}} \put(84,0){\circle*{5}}
\put(84,30){\circle*{5}}\put(44,0){\circle*{5}}
\put(240,0){\circle*{5}} \put(254,0){\circle{5}}
\put(254,30){\circle{5}}\put(218,0){\circle*{5}}
\qbezier[100](0,0)(4,0)(8,0)\qbezier[100](16,0)(20,0)(24,0)
\qbezier[50](0,6)(0,8)(0,14)\qbezier[100](0,22)(0,26)(0,30)
\qbezier[20](20,30)(24,30)(28,30) \qbezier[20](52,30)(56,30)(60,30)
\qbezier[20](68,30)(70,30)(72,30)\qbezier[20](32,0)(36,3)(40,6)
\qbezier[20](48,12)(52,15)(56,18)
\qbezier[20](64,24)(68,27)(72,30)\qbezier[20](4,30)(8,30)(12,30)
\qbezier[20](36,30)(40,30)(44,30) \qbezier[20](84,30)(76,34)(84,38)
\qbezier[20](84,30)(92,34)(84,38)\qbezier[50](44,0)(36,4)(44,8)
\qbezier[50](44,0)(52,4)(44,8)\qbezier[5](254,30)(246,34)(254,38)
\qbezier[5](254,30)(262,34)(254,38)
\qbezier[100](170,0)(174,0)(178,0)\qbezier[100](186,0)(190,0)(194,0)
\qbezier[50](170,6)(170,8)(170,14)\qbezier[100](170,22)(170,26)(170,30)
\qbezier[20](190,30)(194,30)(198,30)
\qbezier[20](222,30)(226,30)(230,30)
\qbezier[20](238,30)(240,30)(242,30)\qbezier[20](202,0)(206,3)(210,6)
\qbezier[20](218,12)(222,15)(226,18)
\qbezier[20](234,24)(238,27)(242,30)\qbezier[20](174,30)(178,30)(182,30)
\qbezier[20](206,30)(210,30)(214,30)
\qbezier(70,0)(62,4)(70,8) \qbezier(70,0)(78,4)(70,8)
\qbezier(84,0)(76,4)(84,8) \qbezier(84,0)(92,4)(84,8)
\qbezier(240,0)(232,4)(240,8) \qbezier(240,0)(248,4)(240,8)
\qbezier[5](254,0)(246,4)(254,8)
\qbezier[5](254,0)(262,4)(254,8)\qbezier(218,0)(210,4)(218,8)
\qbezier(218,0)(226,4)(218,8)
\put(108,15){$\otimes\ \ w\ \ \longmapsto
$}\put(142,22){$\varphi_{2}$} \put(278,15){$\otimes\ \ w$}
\end{picture}
\end{center}
In view of
$
\text{\rm res}_{S_{n-1-t}} (S^\lambda)
\cong \bigoplus_{\alpha\sqsubset\lambda}
S^\alpha
$
we derive
\begin{eqnarray*}
F_n^1(\lambda)\oplus F_n^2(\lambda) & \cong &
\left[\mathbb{I}_{n-1}^{n-1-t}\mathfrak{u}_{n-1,t}\otimes_{S_{(n-1)-t}}
\text{\rm res}_{S_{n-1-t}}(S^\lambda)\right]
\oplus \left[\mathbb{I}_{n-1}^{n-t}\mathfrak{u}_{n-1,t-1} \otimes_{S_{n-t}}
S^\lambda\right] \\
& \cong & \bigoplus_{\alpha\sqsubset\lambda}
\left[\mathbb{I}_{n-1}^{n-1-t}\mathfrak{u}_{n-1,t}\otimes_{S_{(n-1)-t}}
S^\alpha\right]
\oplus \left[\mathbb{I}_{n-1}^{n-t}\mathfrak{u}_{n-1,t-1} \otimes_{S_{n-t}}
S^\lambda\right],
\end{eqnarray*}
which gives rise to the short exact sequence
$
0\longrightarrow \bigoplus_{\alpha\sqsubseteq\lambda}
\mathscr{M}_{\mathbb{A}_{n-1}}(\alpha)\longrightarrow \text{\rm res}_{n-1}
(\mathscr{M}_{\mathbb{A}_n}(\lambda))
$ and Claim $1$ follows.\\
{\it Claim 2}. Let $F_n(\lambda)=F_n^1(\lambda)\oplus F_n^2(\lambda)$,
then we have an isomorphism of
$\mathbb{A}_{n-1}$-left modules
\begin{equation}
\text{\rm res}_{n-1}\left(\mathscr{M}_{\mathbb{A}_n}(\lambda)/F_n(\lambda)
\right)\cong
\bigoplus_{\lambda\sqsubset\beta}\mathscr{M}_{\mathbb{A}_{n-1}}(\beta).
\end{equation}
Let $G_n(\lambda)$ denote the space generated by all
tensors of the form $\mathfrak{c}\otimes w$, where
$\mathfrak{c}\in\mathbb{I}_{n}^{n-t}\mathfrak{u}_{n,t}$ is a diagram with
noncrossing vertical arcs and a horizontal arc incident to $n$.
Let $f_3(\mathfrak{c})$ be the diagram obtained
from $\mathfrak{c}$ as follows:
one removes $n$ together with its incident horizontal arc and the
bottom-vertex $n'$ together with its incident loop. This leaves a
unique top-vertex, $r$, isolated. Next one removes the loop of the
bottom-vertex $(n-t+1)'$ and connects it to $r$ via a vertical arc.
We next show that $f_3$ induces the bijection
\begin{equation}
\begin{split}
\varphi_3\colon
\text{\rm res}_{n-1}\left(\mathscr{M}_{\mathbb{A}_n}(\lambda))/
F_n(\lambda)\right)  &
\longrightarrow
\mathbb{I}_{n-1}^{n+1-t}\mathfrak{u}_{n-1,t-2}
\otimes_{S_{n+1-t}}\text{\rm ind}_{S_{n-t}}^{S_{n+1-t}}(S^\lambda) \\
\mathfrak{c}\otimes w & \longmapsto
f_3(\mathfrak{c})\otimes (w,n+1-t).\qquad\quad\qquad\qquad
\end{split}
\end{equation}
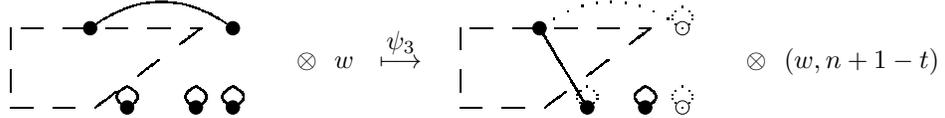
\begin{center}
\begin{picture}(340,20)
\put(70,0){\circle*{5}} \put(84,0){\circle*{5}}
\put(84,30){\circle*{5}}\put(44,0){\circle*{5}}
\put(240,0){\circle*{5}} \put(254,0){\circle{5}}
\put(254,30){\circle{5}}\put(218,0){\circle*{5}}
\put(30,30){\circle*{5}}\put(200,30){\circle*{5}}
\qbezier[100](0,0)(4,0)(8,0)\qbezier[100](16,0)(20,0)(24,0)
\qbezier[50](0,6)(0,8)(0,14)\qbezier[100](0,22)(0,26)(0,30)
\qbezier[20](20,30)(24,30)(28,30) \qbezier[20](52,30)(56,30)(60,30)
\qbezier[20](68,30)(70,30)(72,30)\qbezier[20](32,0)(36,3)(40,6)
\qbezier[20](48,12)(52,15)(56,18)\qbezier[100](30,30)(57,50)(84,30)
\qbezier[20](64,24)(68,27)(72,30)\qbezier[20](4,30)(8,30)(12,30)
\qbezier[20](36,30)(40,30)(44,30) \qbezier[50](44,0)(36,4)(44,8)
\qbezier[50](44,0)(52,4)(44,8)\qbezier[5](254,30)(246,34)(254,38)
\qbezier[5](254,30)(262,34)(254,38)\qbezier[100](218,0)(209,15)(200,30)
\qbezier[100](170,0)(174,0)(178,0)\qbezier[100](186,0)(190,0)(194,0)
\qbezier[50](170,6)(170,8)(170,14)\qbezier[100](170,22)(170,26)(170,30)
\qbezier[20](190,30)(194,30)(198,30)\qbezier[10](200,30)(227,50)(254,30)
\qbezier[20](222,30)(226,30)(230,30)
\qbezier[20](238,30)(240,30)(242,30)\qbezier[20](202,0)(206,3)(210,6)
\qbezier[20](218,12)(222,15)(226,18)
\qbezier[20](234,24)(238,27)(242,30)\qbezier[20](174,30)(178,30)(182,30)
\qbezier[20](206,30)(210,30)(214,30)
\qbezier(70,0)(62,4)(70,8) \qbezier(70,0)(78,4)(70,8)
\qbezier(84,0)(76,4)(84,8) \qbezier(84,0)(92,4)(84,8)
\qbezier(240,0)(232,4)(240,8) \qbezier(240,0)(248,4)(240,8)
\qbezier[5](254,0)(246,4)(254,8)
\qbezier[5](254,0)(262,4)(254,8)\qbezier[5](218,0)(210,4)(218,8)
\qbezier[5](218,0)(226,4)(218,8)
\put(108,15){$\otimes\ \ w\ \ \longmapsto $}\put(142,22){$\psi_{3}$}
\put(278,15){$\otimes\ \ (w,n+1-t)$}
\end{picture}
\end{center}
Recall that for any $1\le j\le n-t$, $\tau_j=(j,n+1-t)$ and $\tau_{n+1-t}=1$.
Then $S_{n+1-t}=\dot\bigcup\tau_jS_{n-t}$, i.e.~the $\tau_r$ form a set
of representatives of $S_{n+1-t}/S_{n-t}$.
We inspect that there exists some $\sigma\in S_{n-t+1}$ such that
$f_3(\mathfrak{c})\sigma^{-1}=\tilde{\mathfrak{c}}$ has noncrossing
vertical arcs. Then we have $\sigma=\tau_j\sigma_0$, for some
$\sigma_0\in S_{n-t}$.
Therefore, in view of $f_3(\mathfrak{c})\sigma^{-1}=
\tilde{\mathfrak{c}}$, each $f_3(\mathfrak{c})$ gives rise to
some unique $\tau_j$. Using eq.~(\ref{E:ind-act}) we obtain
\begin{eqnarray*}
f_3(\mathfrak{c})\sigma^{-1}\sigma\otimes (w,n+1-t) &= &
\tilde{\mathfrak{c}}\tau_j\sigma_0\otimes (w,n+1-t) \\
 & = & \tilde{\mathfrak{c}}\tau_j\otimes (\sigma_0w,n+1-t) \\
& = & \tilde{\mathfrak{c}}\otimes (\sigma_0w,j).
\end{eqnarray*}
There exist exactly $(n+1-t)$ different $\mathscr{I}_{n}^{n-t}$-diagrams
$\mathfrak{c}_1,\dots, \mathfrak{c}_{n+1-t}$ having noncrossing vertical
arcs in which $n$ is connected to a top-vertex and $\text{\rm bot}(
\mathfrak{c}_j)=\text{\rm bot}(\mathfrak{u}_{n,t})$ with the property
\begin{equation}
f_3(\mathfrak{c}_j)\sigma^{(j)}=\tilde{\mathfrak{c}}
\end{equation}
for some $\sigma^{(j)}\in S_{n-t+1}$.
Since $\text{\rm dim}[\text{\rm ind}_{S_{n-t}}^
{S_{n+1-t}}(S^\lambda)]=(n+1-t)\cdot\text{\rm dim}[S^\lambda]$, we obtain
\begin{equation}
\text{\rm dim}\left[\text{\rm res}_{n-1}\left(
\mathscr{M}_{\mathbb{A}_n}(\lambda)/F_n(\lambda)\right)\right]=
\text{\rm dim}\left[\mathbb{I}_{n-1}^{n+1-t}\mathfrak{u}_{n-1,t-2}
\otimes_{S_{n+1-t}}\text{\rm ind}_{S_{n-t}}^{S_{n+1-t}}(S^\lambda)\right].
\end{equation}
Therefore it suffices to prove that $\varphi_3$ is surjective.
$\mathbb{I}_{n-1}^{n+1-t}\mathfrak{u}_{n-1,t-2}
\otimes_{S_{n+1-t}}\text{\rm ind}_{S_{n-t}}^{S_{n+1-t}}(S^\lambda)$ is
generated by tensors of the form $\mathfrak{d}\otimes
(w,j)$, where $1\le j\le n+1-t$, $\mathfrak{d}\in
\mathscr{I}_{n-1}^{n-t+1}$ with noncrossing vertical arcs,
$\text{\rm bot}(\mathfrak{d})=\text{\rm bot}(\mathfrak{u}_{n-1,t-2})$
and $w\in S^\lambda$. Since for $1\le j\le n+1-t$, we have
$\tau_j \cdot (w,n+1-t) = (w,j)$ we obtain
\begin{equation}
\mathfrak{d}\otimes (w,j)= \mathfrak{d}\otimes
\tau_j \cdot (w,n+1-t)= \mathfrak{d}\tau_j \otimes (w,n+1-t).
\end{equation}
By construction $\mathfrak{d}\tau_j$ is a diagram in which $(n+1-t)'$
connected to a top vertex, which we denote by $r$.
Then there exists some $\sigma_0\in S_{n-t}$ such that in $\mathfrak{d}
\tau_j\sigma_0$ any pair of crossing verticals contains the vertical arc
$((n+1-t)',r)$.
Let $\mathfrak{c}\in \mathscr{I}_{n}^{n-t}$, be derived from
$\mathfrak{d}\tau_j\sigma_0$ by removing $(r,(n+1-t)')$,
adding the vertices $n$ and $n'$, the loops at $(n+1-t)'$
and $n'$, as well as the horizontal arc $(r,n)$.
By construction $\text{\rm bot}(\mathfrak{c})=
\text{\rm bot}(\mathfrak{u}_{n,t})$,
$\mathfrak{c}$ has noncrossing verticals and we have
\begin{equation}
(\mathfrak{d}\tau_j)\sigma_0=f_3(\mathfrak{c}).
\end{equation}
Consequently, using the fact that the tensor product is over
$S_{n+1-t}$
\begin{eqnarray*}
\mathfrak{d} \otimes (w,j) &=&
\mathfrak{d}\tau_j \otimes (w,n+1-t)\\
& = & f_3(\mathfrak{c})\sigma_0^{-1} \otimes (w,n+1-t) \\
& = & f_3(\mathfrak{c})\otimes(\sigma_0^{-1}w,n+1-t),
\end{eqnarray*}
which proves that $\varphi_3$ is surjective.
We proceed by showing that $\varphi_3$ is an
isomorphism of $\mathbb{A}_{n-1}$-modules. Since any $\sigma\in S_{n-1}$
fixes $n$ we inspect
\begin{equation}
\forall\;\sigma\in S_{n-1} ;\quad
\varphi_3(\sigma \cdot \mathfrak{c}\otimes w) =
\sigma\cdot f_3(\mathfrak{c})\otimes (w,n+1-t)=
\sigma\cdot \varphi_3(\mathfrak{c}\otimes w).
\end{equation}
We next consider the action of $\mathfrak{e}_i$, $1\le i\le n-2$.
Suppose $n$ is connected to $r$ in $\mathfrak{c}$ and $r\neq i+1,i$.
Then we immediately obtain
\begin{equation}
\varphi_3(\mathfrak{e}_i \cdot \mathfrak{c}\otimes w) =
\mathfrak{e}_i\cdot f_3(\mathfrak{c})\otimes (w,n+1-t)=
\mathfrak{e}_i\cdot \varphi_3(\mathfrak{c}\otimes w).
\end{equation}
Without loss of generality we may assume $r=i$. We distinguish three
cases: \\
(1) if $i+1$ is incident to a vertical arc, in $\mathfrak{e}_i
\mathfrak{c}$ the top-vertex $n$ is connected to a bottom vertex,
whence $\mathfrak{e}_i \cdot \mathfrak{c}\otimes w\equiv 0$ modulo
$F_n(\lambda)$,
\begin{center}
\setlength{\unitlength}{0.75pt}
\begin{picture}(290,50)
\put(10,0){\circle*{5}}\put(36,20){\circle*{5}}\put(52,20){\circle*{5}}
\put(80,20){\circle*{5}}\put(36,40){\circle*{5}}\put(52,40){\circle*{5}}
\put(80,40){\circle*{5}}\put(36,60){\circle*{5}}\put(52,60){\circle*{5}}
\put(80,60){\circle*{5}}\put(150,20){\circle*{5}}\put(166,40){\circle*{5}}
\put(182,40){\circle*{5}}\put(210,40){\circle*{5}}
\qbezier(0,0)(-10,30)(0,60)\qbezier(90,0)(100,30)(90,60)
\qbezier(140,10)(130,30)(140,50)
\qbezier(230,10)(240,30)(230,50)\qbezier(52,60)(44,75)(36,60)
\qbezier(52,40)(44,55)(36,40)
\qbezier(80,40)(80,50)(80,60)\qbezier[5](36,20)(36,30)(36,40)
\qbezier[5](52,20)(52,30)(52,40)\qbezier[5](80,20)(80,30)(80,40)
\qbezier(10,0)(31,10)(52,20)\qbezier(166,40)(174,55)(182,40)
\qbezier(150,20)(180,30)(210,40)\qbezier(36,20)(58,40)(80,20)
\put(-20,30){$f_3$}\put(100,30){$=$}\put(120,30){$f_3$}
\put(31,70){$i$}\put(42,70){$i+1$}\put(78,70){$n$} \put(164,50){$i$}
\put(178,50){$i+1$}\put(208,50){$n$}
\end{picture}
\begin{picture}(190,-30)
\setlength{\unitlength}{0.75pt}
\put(10,20){\circle*{5}}\put(36,40){\circle*{5}}\put(52,40){\circle*{5}}
\put(80,40){\circle*{5}}\put(150,0){\circle*{5}}\put(166,40){\circle*{5}}
\put(182,40){\circle*{5}}\put(166,60){\circle*{5}}
\put(182,60){\circle*{5}}\put(166,20){\circle*{5}}
\put(182,20){\circle*{5}}\put(194,0){\circle*{5}}
\qbezier(10,20)(31,30)(52,40)\qbezier(166,40)(174,55)(182,40)
\qbezier(0,20)(-6,30)(0,40)\qbezier(166,60)(174,75)(182,60)
\qbezier(90,20)(96,30)(90,40)\qbezier(36,40)(58,60)(80,40)
\qbezier(150,0)(166,10)(182,20)
\qbezier(166,20)(180,10)(194,0)\qbezier[8](166,20)(166,30)(166,40)
\qbezier[8](182,20)(182,30)(182,40)
\put(110,30){$=$} \put(31,50){$i$}\put(42,50){$i+1$}\put(78,50){$n$}
\put(164,70){$i$}
\put(178,70){$i+1$}\put(-40,28){$\mathfrak{e}_i\cdot f_{3}$}
\end{picture}
\end{center}
On the other hand, in $f_3(\mathfrak{c})$, $i+1$ and $i$ are
connected to vertical arcs, whence $\mathfrak{e}_i\cdot
\varphi_3(\mathfrak{c}\otimes w)$ has fewer than $(n+1-t)$ vertical
arcs and is consequently zero in
$\mathbb{I}_{n-1}^{n+1-t}\mathfrak{u}_{n-1,t-2}$. \\
(2) if $i+1$ is
incident to a loop, $n$ is incident to a loop in $\mathfrak{e}_i
\mathfrak{c}$. Clearly we then have $f_3(\mathfrak{e}_i
\mathfrak{c})=\mathfrak{e}_if_3(\mathfrak{c})$ implying
$$
\varphi_3(\mathfrak{e}_i \cdot \mathfrak{c}\otimes w) =
\mathfrak{e}_i\cdot f_3(\mathfrak{c})\otimes (w,n+1-t)=
\mathfrak{e}_i\cdot \varphi_3(\mathfrak{c}\otimes w).
$$

\begin{center}
\setlength{\unitlength}{0.75pt}
\begin{picture}(290,50)
\put(80,0){\circle*{5}}\put(36,20){\circle*{5}}\put(52,20){\circle*{5}}
\put(80,20){\circle*{5}}\put(36,40){\circle*{5}}\put(52,40){\circle*{5}}
\put(80,40){\circle*{5}}\put(36,60){\circle*{5}}\put(52,60){\circle*{5}}
\put(80,60){\circle*{5}}\put(210,20){\circle*{5}}\put(166,40){\circle*{5}}
\put(182,40){\circle*{5}}\put(210,40){\circle*{5}}
\qbezier(0,0)(-10,30)(0,60)\qbezier(90,0)(100,30)(90,60)
\qbezier(140,10)(130,30)(140,50)
\qbezier(230,10)(240,30)(230,50)\qbezier(52,60)(44,75)(36,60)
\qbezier(52,40)(44,55)(36,40)
\qbezier(80,40)(80,50)(80,60)\qbezier[5](36,20)(36,30)(36,40)
\qbezier[5](52,20)(52,30)(52,40)\qbezier[5](80,20)(80,30)(80,40)
\qbezier(166,40)(174,55)(182,40)\qbezier(52,20)(44,24)(52,28)
\qbezier(52,20)(60,24)(52,28)\qbezier(80,0)(72,4)(80,8)
\qbezier(80,0)(88,4)(80,8) \qbezier(36,20)(58,40)(80,20)
\qbezier(210,40)(202,44)(210,48)\qbezier(210,40)(218,44)(210,48)
\qbezier(210,20)(202,24)(210,28)\qbezier(210,20)(218,24)(210,28)
\put(-20,30){$f_3$}\put(100,30){$=$}\put(120,30){$f_3$}
\put(31,70){$i$}\put(42,70){$i+1$}\put(78,70){$n$} \put(164,50){$i$}
\put(178,50){$i+1$}\put(208,50){$n$}
\end{picture}
\setlength{\unitlength}{0.75pt}
\begin{picture}(190,-30)
\put(80,20){\circle*{5}}\put(36,40){\circle*{5}}\put(52,40){\circle*{5}}
\put(80,40){\circle*{5}}\put(166,40){\circle*{5}}
\put(182,40){\circle*{5}}\put(166,60){\circle*{5}}
\put(182,60){\circle*{5}}\put(166,20){\circle*{5}}
\put(182,20){\circle*{5}}\put(194,0){\circle*{5}}
\qbezier(166,40)(174,55)(182,40)
\qbezier(0,20)(-6,30)(0,40)\qbezier(166,60)(174,75)(182,60)
\qbezier(90,20)(96,30)(90,40)\qbezier(36,40)(58,60)(80,40)
\qbezier(166,20)(180,10)(194,0)\qbezier[8](166,20)(166,30)(166,40)
\qbezier[8](182,20)(182,30)(182,40)\qbezier(52,40)(44,44)(52,48)
\qbezier(52,40)(60,44)(52,48)\qbezier(80,20)(72,24)(80,28)
\qbezier(182,20)(174,24)(182,28)\qbezier(182,20)(190,24)(182,28)
\qbezier(80,20)(88,24)(80,28)
\put(110,30){$=$} \put(31,50){$i$}\put(42,50){$i+1$}\put(78,50){$n$}
\put(164,70){$i$}
\put(178,70){$i+1$}\put(-40,28){$\mathfrak{e}_i\cdot f_{3}$}
\end{picture}
\end{center}
(3) if $i+1$ is incident to $j$ via a horizontal arc, $n$ is connected
to $j$ in $\mathfrak{e}_i \mathfrak{c}$. Clearly we then have
$f_3(\mathfrak{e}_i \mathfrak{c})=\mathfrak{e}_if_3(\mathfrak{c})$
implying
$$
\varphi_3(\mathfrak{e}_i \cdot \mathfrak{c}\otimes w) =
\mathfrak{e}_i\cdot f_3(\mathfrak{c})\otimes (w,n+1-t)=
\mathfrak{e}_i\cdot \varphi_3(\mathfrak{c}\otimes w).
$$
\begin{center}
\begin{picture}(220,50)
\setlength{\unitlength}{0.75pt}
\put(80,0){\circle*{5}}\put(36,20){\circle*{5}}\put(52,20){\circle*{5}}
\put(80,20){\circle*{5}}\put(36,40){\circle*{5}}\put(52,40){\circle*{5}}
\put(80,40){\circle*{5}}\put(36,60){\circle*{5}}\put(52,60){\circle*{5}}
\put(80,60){\circle*{5}}\put(210,20){\circle*{5}}\put(166,40){\circle*{5}}
\put(182,40){\circle*{5}}\put(210,40){\circle*{5}}\put(16,20){\circle*{5}}
\put(146,40){\circle*{5}}\put(16,40){\circle*{5}}\put(16,60){\circle*{5}}
\qbezier(0,0)(-10,30)(0,60)\qbezier(90,0)(100,30)(90,60)
\qbezier[8](16,40)(16,50)(16,60)
\qbezier(140,10)(130,30)(140,50)\qbezier(146,40)(178,60)(210,40)
\qbezier(230,10)(240,30)(230,50)\qbezier(52,60)(44,75)(36,60)
\qbezier(52,40)(44,55)(36,40)\qbezier(16,20)(34,35)(52,20)
\qbezier(80,40)(80,50)(80,60)\qbezier[5](36,20)(36,30)(36,40)
\qbezier[5](52,20)(52,30)(52,40)\qbezier[5](80,20)(80,30)(80,40)
\qbezier(166,40)(174,55)(182,40)\qbezier(80,0)(72,4)(80,8)
\qbezier(80,0)(88,4)(80,8) \qbezier(36,20)(58,40)(80,20)
\qbezier(210,20)(202,24)(210,28)\qbezier(210,20)(218,24)(210,28)
\put(-20,30){$f_3$}\put(105,30){$=$}\put(120,30){$f_3$}
\put(31,70){$i$}\put(42,70){$i+1$}\put(78,70){$n$} \put(164,70){$i$}
\put(178,70){$i+1$}\put(208,70){$n$}\put(16,70){$j$}
\end{picture}
\begin{picture}(160,-30)
\setlength{\unitlength}{0.75pt}
\put(80,20){\circle*{5}}\put(36,40){\circle*{5}}\put(52,40){\circle*{5}}
\put(80,40){\circle*{5}}\put(166,40){\circle*{5}}
\put(182,40){\circle*{5}}\put(166,60){\circle*{5}}
\put(182,60){\circle*{5}}\put(166,20){\circle*{5}}
\put(182,20){\circle*{5}}\put(194,0){\circle*{5}}\put(16,40){\circle*{5}}
\put(140,20){\circle*{5}}\put(140,40){\circle*{5}}\put(140,60){\circle*{5}}
\qbezier(166,40)(174,55)(182,40)\qbezier(140,40)(140,50)(140,60)
\qbezier(0,20)(-6,30)(0,40)\qbezier(166,60)(174,75)(182,60)
\qbezier(90,20)(96,30)(90,40)\qbezier(36,40)(58,60)(80,40)
\qbezier(166,20)(180,10)(194,0)\qbezier[8](140,20)(140,30)(140,40)
\qbezier[8](166,20)(166,30)(166,40)
\qbezier[8](182,20)(182,30)(182,40)\qbezier(80,20)(72,24)(80,28)
\qbezier(80,20)(88,24)(80,28)\qbezier(16,40)(34,55)(52,40)
\qbezier(140,20)(161,35)(182,20)
\put(110,30){$=$} \put(31,70){$i$}\put(42,70){$i+1$}\put(78,70){$n$}
\put(164,70){$i$}\put(140,70){$j$}
\put(178,70){$i+1$}\put(-40,28){$\mathfrak{e}_i\cdot f_{3}$}
\end{picture}
\end{center}

Finally we consider the action of $\mathfrak{u}_i$, $1\le i\le n-1$.
Suppose first $r\neq i$. By definition of $f_3$, a vertex $i\neq r$
is in $\mathfrak{c}$ incident to a vertical arc if and only if this
holds for $f_3(\mathfrak{c})$.
In this case we have $\mathfrak{u}_i\mathfrak{c}\equiv 0\mod F_n(\lambda)$ and
$\mathfrak{u}_if_3(\mathfrak{c})\equiv 0$ in $\mathbb{I}_{n-1}^{n+1-t}
\mathfrak{u}_{n-1,t-2}$.
If $i$ is incident to a loop we have $\mathfrak{u}_i\mathfrak{c}=x
\mathfrak{c}$ and $\mathfrak{u}_if_3(\mathfrak{c})=xf_3(\mathfrak{c})$,
i.e.~$\varphi_3(\mathfrak{u}_i \cdot \mathfrak{c}\otimes w) =
\mathfrak{u}_i \cdot\varphi_3(\mathfrak{c}\otimes w)$. Finally, if
$i$ is incident to a horizontal arc we have
$f_3(\mathfrak{u}_i\mathfrak{c})=\mathfrak{u}_if_3(\mathfrak{c})$.
Second let $r=i$. On the one hand we obtain $\mathfrak{u}_i\mathfrak{c}
\equiv 0$ modulo $F_n(\lambda)$, since the $i'$-loop of $\mathfrak{u}_i$
traces back to the top vertex $n$ of $\mathfrak{u}_i\mathfrak{c}$.
On the other hand,
in $\mathfrak{u}_if_3(\mathfrak{c})$ the $i'$-loop of $\mathfrak{u}_i$
traces back to the bottom vertex $(n+1-t)'$. Consequently, $\mathfrak{u}_i
f_3(\mathfrak{c})$ has fewer than $(n+1-t)$ vertical arcs and is
zero in $\mathbb{I}_{n-1}^{n+1-t}\mathfrak{u}_{n-1,t-2}$.
Therefore $\varphi_3$ is an isomorphism of $\mathbb{A}_{n-1}$-left modules.
In view of $\text{\rm ind}_{S_{n-t}}^{S_{n+1-t}}(S^\lambda)\cong
\bigoplus_{\lambda\sqsubset\beta}S^\beta$ we derive
\begin{eqnarray*}
\text{\rm res}_{n-1}\left(\mathscr{M}_{\mathbb{A}_n}(\lambda))/
F_n(\lambda)\right) &
\cong & \mathbb{I}_{n-1}^{n+1-t}\mathfrak{u}_{n-1,t-2}
\otimes_{S_{n+1-t}}\text{\rm ind}_{S_{n-t}}^{S_{n+1-t}}(S^\lambda)\\
&\cong & \bigoplus_{\lambda\sqsubset\beta}\left(
\mathbb{I}_{n-1}^{n+1-t}\mathfrak{u}_{n-1,t-2}
\otimes_{S_{n+1-t}}S^\beta\right) \\
& \cong & \bigoplus_{\lambda\sqsubset\beta}\mathscr{M}_{\mathbb{A}_{n-1}}
(\beta)
\end{eqnarray*}
and the proof of the theorem is complete.
\end{proof}

For $\mathbb{L}_n$, there exists no nontrivial space $G_n(\lambda)$ and
Theorem~\ref{T:realdeal} accordingly implies

\begin{corollary}\label{C:realdeal}
Let $n,t\in\mathbb{N}$ and $\lambda\vdash (n-t)$ where $1\le  t
\le n$. Then we have the isomorphism of
$\mathbb{L}_{n-1}$-modules
\begin{eqnarray}\label{E:res-L}
\bigoplus_{\alpha\sqsubseteq \lambda }
\mathscr{M}_{\mathbb{L}_{n-1}}(\alpha)
\cong \text{\rm res}_{n-1}(\mathscr{M}_{\mathbb{L}_n}(\lambda)).
\end{eqnarray}
\end{corollary}
We proceed by studying induction in $\mathbb{A}_n$. Let us begin by
remarking that the arguments of the following proof can easily be
put into context with the localization and globalization functors
\cite{Martin,Doran}. Since the latter are compatible with the
quasi-hereditary structure of $\mathbb{A}_n$, in case of $x\neq 0$
one can obtain a more structural point of view.
\begin{theorem}\label{T:realdeal2}
Let $n,t\in\mathbb{N}$ and $\lambda\vdash (n-t)$ where $1\le  t \le n$. 
Then we have 
\begin{eqnarray}\label{E:indres1}
\text{\rm ind}_{n+1}(\mathscr{M}_{\mathbb{A}_n}(\lambda)) \cong
\text{\rm res}_{n+1}(\mathscr{M}_{\mathbb{A}_{n+2}}(\lambda)).
\end{eqnarray}
Furthermore there exists the exact sequence of $\mathbb{A}_{n+1}$-modules
\begin{eqnarray}
\label{E:induct}
\quad 0\longrightarrow \bigoplus_{\alpha\sqsubseteq \lambda}
\mathscr{M}_{\mathbb{A}_{n+1}}(\alpha)\longrightarrow
\text{\rm ind}_{n+1}(\mathscr{M}_{\mathbb{A}_n}(\lambda))
\longrightarrow 
\bigoplus_{\lambda\sqsubset \beta}\mathscr{M}_{\mathbb{A}_{n+1}}
(\beta)
\longrightarrow 0.
\end{eqnarray}
\end{theorem}
\begin{proof}
We first prove eq.~(\ref{E:indres1}). Suppose we have
$\mathfrak{a}\in \mathscr{A}_{n+2}$, with the property that its
bottom vertices $(n+1)'$ and $(n+2)'$ are connected by a horizontal
arc. Let $f_4(\mathfrak{a})$ be the diagram obtained from
$\mathfrak{a}$ by removing its bottom vertices $(n+1)'$, $(n+2)'$
together with their horizontal arc and moving its top-vertex $(n+2)$
to the bottom at position $(n+1)'$. It is straightforward to prove
that for any $\mathscr{M}_{\mathbb{A}_n}(\lambda)$ the mapping
\begin{equation}\label{E:4}
\begin{split}
\varphi_4\colon
\text{\rm res}_{n+1}(\mathbb{A}_{n+2}\mathfrak{e}_{n+1}
\otimes_{\mathbb{A}_{n}}\mathscr{M}_{\mathbb{A}_n}(\lambda))
& \longrightarrow
 \mathbb{A}_{n+1}\otimes_{\mathbb{A}_n}\mathscr{M}_{\mathbb{A}_n}(\lambda)\\
\mathfrak{a}\otimes w & \longmapsto
f_4(\mathfrak{a})\otimes w,
\end{split}
\end{equation}
is an isomorphism of $\mathbb{A}_{n+1}$-modules. We proceed by showing
\begin{equation}\label{E:interpret}
\mathbb{A}_{n+2}\mathfrak{e}_{n+1}\otimes_{\mathbb{A}_{n}}
\mathscr{M}_{\mathbb{A}_n}(\lambda)\cong \mathscr{M}_{\mathbb{A}_{n+2}}
(\lambda).
\end{equation}
The key to eq.~(\ref{E:interpret}) is to prove that
\begin{equation}\label{E:key}
\mathbb{A}_{n+2}\mathfrak{e}_{n+1}
\otimes_{\mathbb{A}_n}\mathbb{I}_n^{n-t}\mathfrak{u}_{n,t}
\cong
\mathbb{I}_{n+2}^{n-t}\mathfrak{u}_{n+2,t+2}
\end{equation}
is an isomorphism of $\mathbb{A}_{n+2}$-left modules.
For this purpose we consider a tensor $\mathfrak{a}\mathfrak{e}_{n+1}
\otimes \mathfrak{b}\mathfrak{u}_{n,t}$, where
$\mathfrak{a}\mathfrak{e}_{n+1}\in \mathbb{A}_{n+2}\mathfrak{e}_{n+1}$
and $\mathfrak{b}\in\mathscr{A}_{n}^{n-t}$. Let
$\mathfrak{x}\in\mathscr{A}_n^{n-t}$ be obtained from $\mathfrak{b}$ as
follows: we set
$\text{\rm bot}(\mathfrak{x})=\text{\rm top}(\mathfrak{b})$,
$\text{\rm top}(\mathfrak{x})=\text{\rm top}(\mathfrak{u}_{n,t})$
and choose the
vertical $\mathfrak{x}$-arcs and $m\in\mathbb{Z}$ such that
\begin{equation}\label{E:JJ}
x^m\,\mathfrak{x}\,\mathfrak{b}\,\mathfrak{u}_{n,t}=\mathfrak{u}_{n,t}.
\end{equation}
Since the product $\mathfrak{x}^*\mathfrak{x}$ generates exactly $t$
inner components, we obtain using eq.~(\ref{E:JJ})
\begin{equation*}
x^{-t+m}\,\mathfrak{x}^*\mathfrak{x}\,
\mathfrak{b}\,\mathfrak{u}_{n,t}=\mathfrak{b}\,\mathfrak{u}_{n,t}.
\end{equation*}
Using $\text{\rm bot}(\mathfrak{r}^*)=
\text{\rm top}(\mathfrak{x})=\text{\rm top}(\mathfrak{u}_{n,t})$,
we compute
\begin{eqnarray*}
\mathfrak{a}\mathfrak{e}_{n+1}\otimes \mathfrak{b}\,\mathfrak{u}_{n,t} &= &
\mathfrak{a}\mathfrak{e}_{n+1}\otimes
x^{-t+m}\,\mathfrak{x}^*\mathfrak{x}\,
\mathfrak{b}\,\mathfrak{u}_{n,t}\\
 & = &
\mathfrak{a}\mathfrak{e}_{n+1}x^{-t}\,\mathfrak{x}^*\otimes x^m\,\mathfrak{x}\,
\mathfrak{b}\,\mathfrak{u}_{n,t} \\
 & = &
\mathfrak{a}\mathfrak{e}_{n+1}x^{-t}\,\mathfrak{x}^*\otimes \mathfrak{u}_{n,t}
\\
 & = &
x^{-t}\,\mathfrak{a}\,\mathfrak{x}^*\mathfrak{e}_{n+1}\otimes
\mathfrak{u}_{n,t}
\\
 & = & \mathfrak{a}'\mathfrak{u}_{n,t}\mathfrak{e}_{n+1}\otimes
\mathfrak{u}_{n,t}.
\end{eqnarray*}
Employing the  just derived normal form for tensors, we are now in position to
make the isomorphism of $\mathbb{A}_{n+2}$-left modules of eq.~(\ref{E:key})
explicit
\begin{equation*}
\begin{split}
\varphi_5\colon \mathbb{A}_{n+2}\mathfrak{e}_{n+1}\otimes_{\mathbb{A}_n}
\mathbb{I}_n^{n-t}\mathfrak{u}_{n,t}
& \longrightarrow
\mathbb{I}_{n+2}^{n-t}\mathfrak{u}_{n+2,t+2} \\
 \mathfrak{a}'\mathfrak{u}_{n,t}\mathfrak{e}_{n+1}\otimes
\mathfrak{u}_{n,t} & \longmapsto \mathfrak{a}'\mathfrak{u}_{n+2,t+2}.
\end{split}
\end{equation*}
Standard tensor identities imply
\begin{eqnarray*}
\mathbb{A}_{n+2}\mathfrak{e}_{n+1}\otimes_{\mathbb{A}_{n}}
\mathscr{M}_{\mathbb{A}_n}(\lambda) &\cong &
(\mathbb{A}_{n+2}\mathfrak{e}_{n+1}\otimes_{\mathbb{A}_n}
\mathbb{I}_n^{n-t}\mathfrak{u}_{n,t})\otimes_{S_{n-t}} S^\lambda\\
&\cong &
\mathbb{I}_{n+2}^{n-t}\mathfrak{u}_{n+2,t+2}\otimes_{S_{n-t}} S^\lambda\\
& \cong & \mathscr{M}_{\mathbb{A}_{n+2}}(\lambda).
\end{eqnarray*}
Now Claim $3$ follows immediately
\begin{eqnarray*}
\text{\rm ind}_{n+1}(\mathscr{M}_{\mathbb{A}_n}(\lambda)) & = &
\mathbb{A}_{n+1}\otimes_{\mathbb{A}_n} \mathscr{M}_{\mathbb{A}_n}(\lambda) \\
&\cong & \text{\rm res}_{n+1}(\mathbb{A}_{n+2}\mathfrak{e}_{n+1}
\otimes_{\mathbb{A}_{n}}\mathscr{M}_{\mathbb{A}_n}(\lambda))\\
&\cong & \text{\rm res}_{n+1}(\mathscr{M}_{\mathbb{A}_{n+2}}(\lambda)).
\end{eqnarray*}
Accordingly, the exact sequence of eq.~(\ref{E:induct}) is immediately
implied by Theorem~\ref{T:realdeal} and the proof of the theorem is complete.
\end{proof}

\begin{corollary}\label{C:realdeal2}
Let $n,t\in\mathbb{N}$ and $\lambda\vdash (n-t)$ where $1\le  t \le n$. 
Then we have 
\begin{equation}
\text{\rm ind}_{n+1}(\mathscr{M}_{\mathbb{L}_n}(\lambda)) \cong
\text{\rm res}_{n+1}(\mathscr{M}_{\mathbb{L}_{n+2}}(\lambda))
\quad \text{\it and}\quad
\bigoplus_{\alpha\sqsubseteq \lambda}
\mathscr{M}_{\mathbb{L}_{n+1}}(\alpha)\cong
\text{\rm ind}_{n+1}(\mathscr{M}_{\mathbb{L}_n}(\lambda)).
\end{equation}
\end{corollary}


\section{Semisimplicity}
\label{S:semisimple}

The semisimplicity of $\mathbb{L}_n$ is an immediate consequence of
Proposition~\ref{P:cool} and Proposition~\ref{P:cool2}.

\begin{theorem}
Suppose $x\neq 0$, then $\mathbb{L}_n$ is semisimple.
\end{theorem}
\begin{proof}
We showed in Proposition~\ref{P:cool2}, that if $\mathbb{L}_n$ is not 
semisimple, then there exist two partitions $\mu,\lambda$, where
$\vert \mu\vert<\vert\lambda\vert \le n$ and a
nontrivial morphism of $\mathbb{L}_n$-modules
$
\diagram
\mathscr{M}_{\mathbb{L}_n}(\lambda) \rto^{\varphi_n} &
 \mathscr{M}_{\mathbb{L}_n}(\mu).
\enddiagram
$
The uniqueness of $\mathscr{N}_{\mathbb{L}_n}(\mu)$ implies that 
$\varphi_n(\mathscr{M}_{\mathbb{L}_n}
(\lambda))\subset \mathscr{N}_{\mathbb{L}_n}(\mu)$.\\
{\it Claim.} For $x\neq 0$ we have $\mathscr{N}_{\mathbb{L}_n}(\mu)=0$. \\
In case of $\mu\vdash n$ this follows immediately from the irreducibility of
the lift of the Specht module $S^\lambda$. Suppose next $\mu\vdash (n-t)<n$.
Let $\mathfrak{a}\in \mathscr{I}_n^{n-t}$, where $\text{\rm bot}(\mathfrak{a})
=\text{\rm bot}(\mathfrak{u}_{n,t})$ and let $v\in S^\mu$. For any
$\mathfrak{a}\otimes v\in \mathscr{M}_{\mathbb{L}_n}(\mu)$,
there exists some $\sigma_0\in S_{n-t}$ and some $t$-tuple
$(j_1,j_2,\dots,j_t)$, where $1\le j_1<j_2<\dots<j_t\le n$ such
that $\mathfrak{a}_{(j_h)_{h=1}^t}=\mathfrak{a}\sigma_0$ has noncrossing
vertical arcs and has top-vertex loops at $j_1,\dots,j_t$.
The $\mathscr{I}_n^{n-t}$-diagram, $\mathfrak{a}_{(j_h)_{h=1}^t}$ has the 
property
$\mathfrak{a}\otimes v=\mathfrak{a}_{(j_h)_{h=1}^t}\otimes \sigma_0^{-1} v$ and
any $u\in \mathbb{I}_n^{n-t}\mathfrak{u}_{n,t}\otimes_{S_{n-1}}
S^\mu$ can be written as
\begin{equation*}
u=\sum_{1\le j_1<j_2<\dots<j_t\le n}\mathfrak{a}_{(j_h)_{h=1}^t}\otimes
w_{(j_h)_{h=1}^t}.
\end{equation*}
For $U_t=\sum_{1\le i_1<i_2<\dots<i_t\le n}\mathfrak{u}_{i_1}\cdots
\mathfrak{u}_{i_t}\in\mathbb{I}_n^{n-t}$ we immediately obtain
$U_t\cdot \mathfrak{a}_{(j_h)_{h=1}^t}=x^t\,
\mathfrak{a}_{(j_h)_{h=1}^t}$. Indeed, only if
$(i_1,i_2,\dots,i_t)$ matches the tuple $(j_h)_{h=1}^t$
the factor $x^t$ via the $t$-inner components
of the graph $G(\mathfrak{u}_{i_j}\cdots\mathfrak{u}_{i_t},
\mathfrak{a}_{(j_h)_{h=1}^t})$ is produced. In all other cases there exists a
loop which traces back to the bottom row of
$G'(\mathfrak{u}_{i_j}\cdots\mathfrak{u}_{i_t},\mathfrak{a}_{(j_h)_{h=1}^t})$
resulting in a zero in $\mathbb{I}_n^{n-t}$.
Therefore, for any
$u\in \mathscr{M}_{\mathbb{L}_n}(\mu)$
\begin{equation}
U_t\cdot u=x^t\, u
\end{equation}
holds. Since $U_t\in\mathbb{I}_{n}^{n-t}$, $x\neq 0$ implies
$\mathscr{N}_{\mathbb{L}_n}(\mu)=
\{w\in\mathscr{M}_{\mathbb{L}_n}
(\mu)\mid \mathbb{I}_{n}^{n-t}w=0\}=0$ and the Claim is proved.\\
The uniqueness of $\mathscr{N}_{\mathbb{L}_n}(\mu)$ as a maximal
$\mathscr{M}_{\mathbb{L}_n}(\mu)$ module implies that
any nontrivial morphism $\varphi_n$ has the property
$\varphi_n(\mathscr{M}_{\mathbb{L}_n}(\lambda))\subset
\mathscr{N}_{\mathbb{L}_n}(\mu)$.
Therefore we arrive at $\varphi_n(\mathscr{M}_{\mathbb{L}_n}(\lambda))=0$, 
i.e.~there exists no nontrivial morphism $\varphi_n\colon
\mathscr{M}_{\mathbb{L}_n}(\lambda)\longrightarrow 
\mathscr{M}_{\mathbb{L}_n}(\mu)$,
whence $\mathbb{L}_n$ is semisimple.
\end{proof}

We next consider the algebra $\mathbb{A}_n$.
According to Proposition~\ref{P:cool2}, if $\mathbb{A}_n$ is not semisimple
then there exists the exact sequence
\begin{equation}\label{E:murx}
\diagram
0\rto & \mathscr{N}_{\mathbb{A}_n}(\lambda) \rto &
\mathscr{M}_{\mathbb{A}_n}(\lambda) \rto^{\varphi_n}
& \mathscr{M}_{\mathbb{A}_n}(\mu),
\enddiagram
\end{equation}
where $\mu,\lambda$ are two partitions, such that $\lambda\vdash (n-t_\lambda)$
and $\mu\vdash (n-t_\mu)$, $t_\lambda<t_\mu$. In the next lemma we show that
we can, without loss of generality, assume that $\lambda\vdash n$.
Since $\mathscr{M}_{\mathbb{A}_n}(\lambda)\cong S^\lambda$ is irreducible  
this implies that we have 
an embedding $\varphi_n\colon S^\lambda\longrightarrow
\mathscr{M}_{\mathbb{A}_n}(\mu)$.
\begin{lemma}\label{L:one}
Suppose $x\neq 0$ and $\mathbb{A}_n$ is not semisimple.
Then there exists $n_1\le n$, two partitions $\lambda_1\vdash n_1$,
$\mu_1\vdash n_1-t_1$ and the short exact
sequence
\begin{equation}\label{E:short}
\diagram
0\rto  &
S^{\lambda_1} \rto^{\varphi_{n_1}\quad } &
\mathscr{M}_{\mathbb{A}_{n_1}}(\mu_1).
\enddiagram
\end{equation}
\end{lemma}
\begin{proof}
If $\mathbb{A}_n$ is not semisimple, then there exists
$\lambda\vdash (n-t_\lambda)$, $\mu\vdash (n-t_\mu)$, where
$t_\lambda<t_\mu$ and the exact sequence of eq.~(\ref{E:murx}).
Without loss of generality we may assume $0<t_\lambda$.
Theorem~\ref{T:realdeal} guarantees
the existence of the embedding
$e_\lambda\colon\mathscr{M}_{\mathbb{A}_{n-1}}(\lambda)
\longrightarrow \mathscr{M}_{\mathbb{A}_n}(\lambda)$ and
$e_\mu\colon \mathscr{M}_{\mathbb{A}_{n-1}}(\mu) \longrightarrow
\mathscr{M}_{\mathbb{A}_n}(\mu)$ given by
$e_\lambda(\mathfrak{a}\otimes v)=\mathfrak{a}\mathfrak{u}_n\otimes
v$ and $e_\mu(\mathfrak{a}\otimes w)=
\mathfrak{a}\mathfrak{u}_n\otimes w$, respectively. We shall show
that $\varphi_n\colon
\mathscr{M}_{\mathbb{A}_n}(\lambda)\longrightarrow
\mathscr{M}_{\mathbb{A}_n}(\mu)$ induces a nontrivial morphism of
$\mathbb{A}_{n-1}$-left modules via
\begin{equation}\label{E:diagram}
\diagram
\mathscr{M}_{\mathbb{A}_n}(\lambda) \rrto_{\varphi_n} & &
\mathscr{M}_{\mathbb{A}_n}(\mu) \\
\mathscr{M}_{\mathbb{A}_{n-1}}(\lambda)\uto_{e_\lambda}
\ar@/_0pc/@{-->}[rr]^{\varphi_{n-1}}
 & &
\mathscr{M}_{\mathbb{A}_{n-1}}(\mu)\uto_{e_\mu}
\enddiagram
\end{equation}
Let $\mathfrak{a}\otimes v\in \mathscr{M}_{\mathbb{A}_{n-1}}(\lambda)$,
where $\mathfrak{a}\in
\mathscr{I}_{n-1}^{n-t_\lambda}$, $\text{\rm bot}(\mathfrak{a})=\text{\rm
bot}(\mathfrak{u}_{n-1,t_\lambda-1})$ and $v\in S^\lambda$.
Since $\varphi_{n}$ is a morphism of $\mathbb{A}_{n}$-left modules we have
$
\varphi_n(\mathfrak{u}_n\cdot e_\lambda(\mathfrak{a}\otimes v))
=\mathfrak{u}_n\cdot\varphi(e_\lambda(\mathfrak{a}\otimes v))
$
and in view of $
\varphi_n(\mathfrak{u}_n\cdot e_\lambda(\mathfrak{a}\otimes v))=x\,
\varphi_n(e_\lambda(\mathfrak{a}\otimes v))$ we derive
\begin{equation}\label{E:hh}
\mathfrak{u}_n\varphi_n(e_\lambda(\mathfrak{a}\otimes v))=x\,
\varphi_n(e_\lambda(\mathfrak{a}\otimes v)) .
\end{equation}
We represent $\varphi_n(e_\lambda(\mathfrak{a}\otimes v))=
\sum_{r}\mathfrak{a}_r\otimes v_r$, where the $\mathfrak{a}_r$
are distinct $\mathscr{I}_{n}^{n-t_\mu}$-diagrams, having noncrossing
verticals, with $\text{\rm bot}(\mathfrak{a}_r)=\text{\rm
bot}(\mathfrak{u}_{n,t_\mu})$ and $v_r\in S^\mu$. Then we obtain
\begin{equation}\label{E:HJ}
x^{-1}\mathfrak{u}_n\varphi_n(e_\lambda(\mathfrak{a}\otimes v))=
x^{-1}\sum_{r}(\mathfrak{u}_n\mathfrak{a}_r)\otimes v_r=
\sum_{r}\mathfrak{a}_r\otimes v_r.
\end{equation}
Since different $\mathscr{I}_{n}^{n-t_\mu}$-diagrams are by construction
linear independent we can conclude from eq.~(\ref{E:HJ}), that each
$\mathfrak{a}_{r}$ has a loop at top-vertex $n$.
Therefore there exists for each
$\varphi_n(e_\lambda(\mathfrak{a}\otimes v))=\sum_{r}\mathfrak{a}_r
\otimes v_r$, a unique element
$\sum_{r}\mathfrak{a}^\ddagger_r \otimes v_r\in \mathscr{M}_{\mathbb{A}_{n-1}}
(\mu)$,
obtained by removing the vertices $n$ and $n'$ and their corresponding loops
from each $\mathfrak{a}_r$. Since $\mathscr{M}_{\mathbb{A}_{n-1}}(\lambda)$ is
generated
by tensors of the form $\mathfrak{a}\otimes v$, $\varphi_n$ induces the mapping
\begin{equation}\label{E:hhh}
\begin{split}
\varphi_{n-1}\colon \mathscr{M}_{\mathbb{A}_{n-1}}(\lambda) & \longrightarrow
\mathscr{M}_{\mathbb{A}_{n-1}}(\mu)\\
\mathfrak{a}\otimes v & \longmapsto
\sum_{r}\mathfrak{a}^\ddagger_r \otimes v_j ,
\end{split}
\end{equation}
with the property $e_\mu\cdot \varphi_{n-1}=\varphi_n\cdot
e_\lambda$, i.e.~$\varphi_{n-1}$ makes the diagram in
eq.~(\ref{E:diagram}) commutative. By construction, $\varphi_{n-1}$
is a morphism
of $\mathbb{A}_{n-1}$-left modules. \\
{\it Claim.} We have $w\in \mathscr{N}_{\mathbb{A}_{n-1}}(\lambda)$ if and 
only if $e_\lambda(w)\in \mathscr{N}_{\mathbb{A}_n}(\lambda)$.\\
Suppose first $w=\sum_i \mathfrak{a}_i\otimes v_i\not
\in \mathscr{N}_{\mathbb{A}_{n-1}}(\lambda)$. According to
eq.~(\ref{E:b0}), there exists some
$\mathfrak{b}_0\in\mathscr{I}_{n-1}^{n-t_\lambda}$
such that
\begin{equation*}
\mathfrak{b}_0\cdot \sum_i \mathfrak{a}_i\otimes v_i=
\sum_i (\mathfrak{b}_0\mathfrak{a}_i)\otimes v_i=
\widetilde{\mathfrak{b}}_0\otimes \sum_i\delta_{\mathfrak{b}_0\mathfrak{a}_i}
x^{\ell(\mathfrak{b}_0,\mathfrak{a}_i)}
\sigma_{\mathfrak{b}_0,\mathfrak{a}_i}v_i\neq 0.
\end{equation*}
This equation implies in the $\mathbb{A}_n$-module 
$\mathscr{M}_{\mathbb{A}_n}(\lambda)$
\begin{equation*}
\mathfrak{b}_0
\cdot \sum_i \mathfrak{a}_i\mathfrak{u}_n\otimes v_i
 =
\sum_i (\mathfrak{b}_0\mathfrak{u}_n\mathfrak{a}_i)\otimes v_i
\end{equation*}
where $\mathfrak{b}_0\mathfrak{u}_n\in \mathscr{I}_{n}^{n-t_\lambda}$.
In view of $\ell(\mathfrak{b}_0\mathfrak{u}_n,\mathfrak{a}_i)=
\ell(\mathfrak{b}_0,\mathfrak{a}_i)$
and
$
\mathfrak{b}_0\mathfrak{a}_i=x^{\ell(\mathfrak{b}_0,\mathfrak{a}_i)}
\tilde{\mathfrak{b}}_0\sigma_{\mathfrak{b}_0,\mathfrak{a}_i}$,
where $\sigma_{\mathfrak{b}_0,\mathfrak{a}_i}\in S_{n-t_\lambda}$ we obtain
\begin{equation*}
\mathfrak{b}_0
\cdot \sum_i \mathfrak{a}_i\mathfrak{u}_n\otimes v_i
 =
\widetilde{\mathfrak{b}}_0\mathfrak{u}_n
\otimes \sum_i\delta_{\mathfrak{b}_0\mathfrak{a}_i}
x^{\ell(\mathfrak{b}_0,\mathfrak{a}_i)}
\sigma_{\mathfrak{b}_0,\mathfrak{a}_i}v_i\neq 0.
\end{equation*}
I.e.~we have shown $w \not\in \mathscr{N}_{\mathbb{A}_{n-1}}(\lambda)\
\Longrightarrow\ e_\lambda(w)\not \in \mathscr{N}_{\mathbb{A}_n}(\lambda)$.
Second suppose $e_\lambda(w)\not \in \mathscr{N}_{\mathbb{A}_n}(\lambda)$.
Then there exists some $\mathfrak{b}_0\in\mathscr{I}_n^{n-t_\lambda}$ such that
\begin{equation*}
\mathfrak{b}_0
\cdot \sum_i \mathfrak{a}_i\mathfrak{u}_n\otimes v_i=
\sum_i (\mathfrak{b}_0\mathfrak{u}_n)
\mathfrak{a}_i\otimes v_i\neq 0
\end{equation*}
and $\mathfrak{b}_0\mathfrak{u}_n$ is the scalar
multiple of a diagram $\mathfrak{x}\in \mathscr{I}_n^{n-t_\lambda}$,
having $\text{\rm top}(\mathfrak{x})=
\text{\rm top}(\mathfrak{b}_0)$ and a loop at $n'$. We accordingly compute
\begin{equation*}
\mathfrak{b}_0
\cdot \sum_i \mathfrak{a}_i\mathfrak{u}_n\otimes v_i= x^s
\sum_i \mathfrak{x}
\mathfrak{a}_i\otimes v_i=
x^s\,
\widetilde{\mathfrak{r}}
\otimes \sum_i\delta_{\mathfrak{r}\mathfrak{a}_i}
x^{\ell(\mathfrak{r},\mathfrak{a}_i)}
\sigma_{\mathfrak{r},\mathfrak{a}_i}v_i\neq 0,
\end{equation*}
where $\tilde{\mathfrak{r}}$ is given by
$\mathfrak{r}\mathfrak{a}_i=x^{\ell(\mathfrak{r},\mathfrak{a}_i)}
\delta_{\mathfrak{r}\mathfrak{a}_i}
\tilde{\mathfrak{r}}\,
\sigma_{\mathfrak{r},\mathfrak{a}_i}$.
We may assume that $\mathfrak{r}$ has a loop at top-vertex $n$, since this
feature does not affect the term $w_1=\sum_i\delta_{\mathfrak{r}\mathfrak{a}_i}
x^{\ell(\mathfrak{r},\mathfrak{a}_i)}\sigma_{\mathfrak{r},\mathfrak{a}_i}v_i$.
By construction, $\tilde{\mathfrak{r}}$ has then also a loop at $n$ and
there exists a
$\mathfrak{c}\in\mathscr{I}_{n-1}^{n-t_\lambda}$ with the property
$\mathfrak{x}=\mathfrak{c}\mathfrak{u}_n$ in $\mathscr{I}_n^{n-t_\lambda}$.
In view of $\delta_{\mathfrak{x}\mathfrak{a}_i} =
\delta_{\mathfrak{c}\mathfrak{a}_i}$, $\sigma_{\mathfrak{x},\mathfrak{a}_i}=
\sigma_{\mathfrak{c},\mathfrak{a}_i}$ and $\ell(\mathfrak{x},\mathfrak{a}_i)=
\ell(\mathfrak{c},\mathfrak{a}_i)$ we obtain
\begin{equation*}
\mathfrak{c}
\cdot \sum_i \mathfrak{a}_i\otimes v_i
=\sum_i (\mathfrak{c}\mathfrak{a}_i)\otimes v_i=
 \widetilde{\mathfrak{c}}
\otimes \sum_i\delta_{\mathfrak{c}\mathfrak{a}_i}
x^{\ell(\mathfrak{c},\mathfrak{a}_i)}
\sigma_{\mathfrak{c},\mathfrak{a}_i}v_i=
 \widetilde{\mathfrak{c}}\otimes w_1
\neq 0.
\end{equation*}
That is, we have proved $e_\lambda(w)\not \in \mathscr{N}_{\mathbb{A}_n}
(\lambda)
\ \Longrightarrow\ w \not\in \mathscr{N}_{\mathbb{A}_{n-1}}(\lambda)$
and the Claim follows.\\
Using $e_\mu\cdot \varphi_{n-1}=\varphi_n\cdot e_\lambda$, we can
now immediately conclude $\text{\rm
ker}(\varphi_{n-1})=\mathscr{N}_{\mathbb{A}_{n-1}}(\lambda)$.
Indeed, if $w\in \text{\rm ker}(\varphi_{n-1})$ then
$e_\lambda(w)\in \text{\rm
ker}(\varphi_n)=\mathscr{N}_{\mathbb{A}_n}(\lambda)$, whence $w\in
\mathscr{N}_{\mathbb{A}_{n-1}}(\lambda)$. If $w\in
\mathscr{N}_{\mathbb{A}_{n-1}}(\lambda)$, then $e_\lambda(w)\in
\mathscr{N}_{\mathbb{A}_n}(\lambda)$, whence $e_\mu\cdot
\varphi_{n-1}(w)=0$, from which $\varphi_{n-1}(w)=0$,
i.e.~$w\in\text{\rm ker}(\varphi_{n-1})$ follows. Therefore we have
the commutative diagram
\begin{equation*}
\diagram
0\rto &  \mathscr{N}_{\mathbb{A}_{n}}(\lambda)\rto&
\mathscr{M}_{\mathbb{A}_{n}}(\lambda)
\rto^{\varphi_{n}} &  \mathscr{M}_{\mathbb{A}_{n}}(\mu)\\
0\rto &  \mathscr{N}_{\mathbb{A}_{n-1}}(\lambda)\uto^{e_\lambda}\rto&
\mathscr{M}_{\mathbb{A}_{n-1}}(\lambda)\uto^{e_\lambda}
\rto^{\varphi_{n-1}} &  \mathscr{M}_{\mathbb{A}_{n-1}}(\mu)\uto^{e_\mu}.
\enddiagram
\end{equation*}
Iterating the above construction we arrive, after $t_\lambda$ steps,
at some $\lambda_1\vdash n_1$, $\mu_1\vdash (n-t_1)$ and the exact sequence
$
\diagram
0\rto &  \mathscr{N}_{\mathbb{A}_{n_1}}(\lambda_1)\rto&
\mathscr{M}_{\mathbb{A}_{n_1}}(\lambda_1)
\rto^{} &   \mathscr{M}_{\mathbb{A}_{n_1}}(\mu_1)
\enddiagram
$. Since $\lambda_1\vdash n_1$ we have $\mathscr{N}_{\mathbb{A}_{n_1}}
(\lambda_1)=0$ and
$\mathscr{M}_{\mathbb{A}_{n_1}}(\lambda_1)\cong S^{\lambda_1}$, whence
Lemma~\ref{L:one}.
\end{proof}

Our next result establishes further restrictions on the embedding
$\diagram 0 \rto & S^{\lambda} \rto &  \mathscr{M}_{\mathbb{A}_n}(\mu)
\enddiagram$ in terms of the partition $\mu$.
Using the exact sequence for the restriction functor of
Theorem~\ref{T:realdeal} we shall prove the existence of such an
embedding with the additional property $\mu_2\vdash (n-1)$ or
$\mu_2\vdash (n-2)$. The result is in analogy to the Brauer
algebra case proved in \cite{Doran}.
\begin{lemma}\label{L:two}
Suppose $\varphi_{n_1}\colon S^{\lambda_1} \longrightarrow
\mathscr{M}_{\mathbb{A}_{n_1}}(\mu_1)$ is an embedding where
$\lambda_1\vdash n_1$ and $\mu_1\vdash n_1-t_1$. Then
for $\mathbb{A}_n$ there exist $n_2\le n_1$, a pair of
partitions $(\lambda_2,\mu_2)$ and an embedding
$S^{\lambda_2} \longrightarrow \mathscr{M}_{\mathbb{A}_{n_2}}(\mu_2)$,
such that $\lambda_2\vdash n_2$,
$\mu_2\vdash (n_2-1)$ or $\mu_2\vdash (n_2-2)$.
\end{lemma}
\begin{proof}
Since $\text{\rm res}_{S_{n_1-1}}(S^{\lambda_1})\cong  \bigoplus_{\nu
\sqsubset\lambda_1}S^{\nu}$ we obtain for some $\nu\sqsubset\lambda_1$,
$\nu\vdash (n_1-1)$, the embedding $\varphi_\nu\colon S^{\nu}
\longrightarrow \text{\rm res}_{n_1-1}(\mathscr{M}_{\mathbb{A}_{n_1}}(\mu_1))$.
An interpretation of 
$\text{\rm res}_{n-1}(\mathscr{M}_{\mathbb{A}_{n_1}}(\mu_1))$ is
given via Theorem~\ref{T:realdeal} in terms of the exact sequence
\begin{equation}\label{E:ach}
0\longrightarrow \bigoplus_{\alpha_1\sqsubseteq \mu_1 }
\mathscr{M}_{{n_1}-1}(\alpha_1)
\longrightarrow \text{\rm res}_{{n_1}-1}(\mathscr{M}_{\mathbb{A}_{n_1}}(\mu_1))
\longrightarrow \bigoplus_{\mu_1\sqsubset \beta_1}\mathscr{M}_{{n_1}-1}
(\beta_1)
\longrightarrow 0.
\end{equation}
Suppose we have $\varphi_\nu(S^{\nu})\subset  F_{n_1}(\mu_1)\cong
\bigoplus_{\alpha_1\sqsubseteq \mu_1}\mathscr{M}_{{n_1}-1}(\alpha_1)$.
Then the irreducibility of $S^{\nu}$ implies the embedding $S^{\nu}
\longrightarrow \mathscr{M}_{{n}_1-1}(\alpha_1)$ for some $\alpha_1\sqsubseteq
\mu_1$. Otherwise, we have $\varphi_\nu(S^{\nu})\not\subset  F_{n_1}(\mu_1)$.
The irreducibility of $\varphi_\nu(S^{\nu})$ guarantees
\begin{equation*}
\left(\varphi_\nu(S^{\nu})\oplus F_{n_1}(\mu_1)\right)/F_{n_1}(\mu_1)\cong
\varphi_\nu(S^{\nu}).
\end{equation*}
In view of eq.~(\ref{E:ach}) we have
\begin{equation*}
\text{\rm res}_{{n_1}-1}(\mathscr{M}_{\mathbb{A}_{n_1}}
(\mu_1))/F_{n_1}(\mu_1)\cong
\bigoplus_{\mu_1\sqsubset \beta_1}\mathscr{M}_{{n_1}-1}(\beta_1)
\end{equation*}
which implies an embedding $S^{\nu}\longrightarrow
\mathscr{M}_{{n_1}-1}(\beta_1)$, for some $\mu_1\sqsubset \beta_1$.\\
Therefore we have the following situation: each iteration of the
above argument reduces the size of the partition $\lambda_1\vdash
n_1$ by one and an analogous reduction of the partition $\mu_1$ can
occur at most $(n_1-t_1)<n_1$ times. Any further iteration cannot
decrease the size of $\mu_1$, while decreasing the size of
$\lambda_1$. That is, iteration produces a pair $(\lambda_2,\mu_2)$
where $\lambda_2\vdash n_2$ and $\mu_2 \vdash (n_2-h)$, where $h=1$
or $h=2$. 
Indeed, for $h=2$, i.e.~$\mu_2\vdash n_2-2$, further
reduction can generate the trivial embedding $S^\nu\longrightarrow
S^\nu$, i.e.~we derive, using the above notation, $\nu=\beta_1$, for
$\nu\sqsubset \lambda_2$. Therefore further reduction is in general
not possible and we have shown that iteration of the above process
leads to a pair of partitions $(\lambda_2,\mu_2)$ with the
properties $\lambda_2\vdash n_2$ and $\mu_2\vdash (n_2-1)$ or
$\mu_2\vdash (n_2-2)$.
\end{proof}

Now we are in position to prove our main result:
\begin{theorem}\label{T:simple}
Suppose $x\neq 0$. If $x\not\in\mathbb{Z}$, then the algebra $\mathbb{A}_n$
is semisimple.
\end{theorem}
\begin{proof}
According to Proposition~\ref{P:cool2}, if $\mathbb{A}_n$ is not semisimple
there exists a nontrivial morphism $\varphi_n\colon\mathscr{M}_{\mathbb{A}_n}
(\lambda)\longrightarrow \mathscr{M}_{\mathbb{A}_n}(\mu)$ with
$\text{\rm ker}(\varphi_n)=\mathscr{N}_{\mathbb{A}_n}(\lambda)$.\\
In view of Lemma~\ref{L:one} and Lemma~\ref{L:two} we can, without loss of
generality, assume that there exists an embedding $\varphi_n\colon S^{\lambda}
\longrightarrow \mathscr{M}_{\mathbb{A}_n}(\mu)$, where $\lambda\vdash n$ and
either $\mu\vdash (n-1)$ or $\mu\vdash (n-2)$.
According to Proposition~\ref{P:cool}, $\mathscr{N}_{\mathbb{A}_n}(\mu)$
is the unique maximal $\mathscr{M}_{\mathbb{A}_n}(\mu)$-submodule.
Therefore $\varphi_n(S^{\lambda})\subset \mathscr{N}_{\mathbb{A}_n}(\mu)$,
i.e.~we have the embedding $\varphi_n\colon S^\lambda \longrightarrow
\mathscr{N}_{\mathbb{A}_n}(\mu)$.
In the following we distinguish the two cases $\mu\vdash (n-1)$ and
$\mu\vdash (n-2)$.\\
{\it Case $1$: $\mu\vdash (n-1)$.} We prove that $x\neq 0$ implies
$\mathscr{N}_{\mathbb{A}_n}(\mu)=0$. Let $\mathfrak{a}\in
\mathscr{I}_n^{n-1}$, where $\text{\rm bot}(\mathfrak{a}) =\text{\rm
bot}(\mathfrak{u}_{n,1})$ and let $v\in S^\mu$. For any
$\mathfrak{a}\otimes v\in \mathscr{M}_{\mathbb{A}_n}(\mu)$, there
exists some $\sigma_0\in S_{n-1}$ and some index $1\le j\le n$ such
that $\mathfrak{a}_j=\mathfrak{a}\sigma_0$ has noncrossing vertical
arcs and has its unique, top-vertex loop at $j$. $\mathfrak{a}_j$
has the property $\mathfrak{a}\otimes v=\mathfrak{a}_j\otimes
\sigma_0^{-1} v$ and any $u\in
\mathbb{I}_n^{n-1}\mathfrak{u}_{n,1}\otimes_{S_{n-1}}S^\mu$ can be
written as $u=\sum_j\mathfrak{a}_j\otimes w_j$. Let
$U_1=\sum_i\mathfrak{u}_i$. Then $U_1\in\mathbb{I}_n^{n-1}$ and any
$\mathfrak{a}_j$ satisfies the eigenvector equation $U_1\cdot
\mathfrak{a}_j=x\, \mathfrak{a}_j$. Let
$u=\sum_j\mathfrak{a}_j\otimes
w_j\in\mathscr{N}_{\mathbb{A}_n}(\mu)$. Since
$U_1\in\mathbb{I}_n^{n-1}$, the action of $U_1$ on
$\mathscr{N}_{\mathbb{A}_n}(\mu)$ is trivial, i.e.~
\begin{equation}\label{E:Case1}
U_1\cdot  u =  \sum_j(U_1\cdot \mathfrak{a}_j)\otimes w_j =
 \sum_ix\,\mathfrak{a}_j\otimes w_j=
x \, u = 0,
\end{equation}
which implies, in view of $x\neq 0$, $\mathscr{N}_{\mathbb{A}_n}(\mu)=0$. \\
{\it Case $2$: $\mu\vdash (n-2)$.}
For each diagram $\mathfrak{a}\in\mathscr{I}_n^{n-2}$, such that
$\text{\rm bot}(\mathfrak{a})=\text{\rm bot}(\mathfrak{u}_{n,2})$
there exist a pair of indices, $i<j$ and a permutation $\sigma_0\in S_{n-2}$
such that either $\mathfrak{a}\sigma_0=\mathfrak{a}_{i,j}^\cap$ or
$\mathfrak{a}\sigma_0=\mathfrak{a}_{i,j}^{\circ}$ holds.
Here $\mathfrak{a}^\cap_{i,j}\in\mathscr{I}_n^{n-2}$ has noncrossing verticals,
a horizontal arc connecting $i$ and $j$ and $\text{\rm bot}(\mathfrak{a}^
\cap_{i,j})=\text{\rm bot}(\mathfrak{u}_{n,2})$. Analogously,
$\mathfrak{a}^{\circ}_{i,j}\in\mathscr{I}_n^{n-2}$ has noncrossing
verticals, two loops at $i,j$ and $\text{\rm bot}(\mathfrak{u}_{n,2})$.
We can write each tensor $\mathfrak{a}\otimes w$, where $\mathfrak{a}
\in\mathscr{I}_n^{n-2}$ with
$\text{\rm bot}(\mathfrak{a})=\text{\rm bot}(\mathfrak{u}_{n,2})$
and $w\in S^\mu$, uniquely as either
$\mathfrak{a}_{i,j}^{\cap}\otimes \sigma_0^{-1}w$ or
$\mathfrak{a}_{i,j}^{\circ}\otimes \sigma_0^{-1}w$.
Let $g\colon \mathscr{M}_{\mathbb{A}_n}(\mu)\longrightarrow
\mathscr{M}_{\mathbb{A}_n}(\mu)$ be the involution given via linear extension 
of $
g(\mathfrak{a}^{\circ}_{i,j}\otimes w)=\mathfrak{a}^{\cap}_{i,j}\otimes
w$ and $g(\mathfrak{a}^{\cap}_{i,j}\otimes w)=
\mathfrak{a}^{\circ}_{i,j}\otimes w$.
Furthermore, let $\mathfrak{v}_{i,j}\in\mathscr{I}_n^{n-2}$ be the diagram 
having
straight verticals except of a horizontal arc connecting the top-vertices $i,j$
and two loops at the bottom vertices $i',j'$, respectively.
We introduce
\begin{equation}
U_2 =  \sum_{i<j}\mathfrak{u}_i\mathfrak{u}_j,\
V_2 =  \sum_{i<j}\mathfrak{v}_{i,j} \quad\text{\rm and}\quad
H_2 = \sum_{i<j}\mathfrak{h}_{i,j},
\end{equation}
where $\mathfrak{h}_{i,j}\in\mathscr{I}_n^{n-2}$ has straight vertical
arcs except of the top-vertices $i,j$ and bottom-vertices $i',j'$, which
are connected by a horizontal arc, respectively. We observe
$U_2,V_2,H_2\in \mathbb{I}_n^{n-2}$.
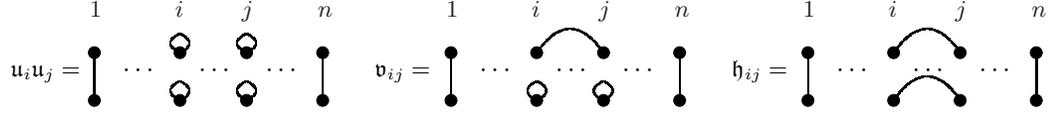
\begin{center}
\setlength{\unitlength}{0,9pt}
\begin{picture}(380,40)
\put(-2,35){\small{$1$}}\put(34,35){\small{$i$}}
\put(62,35){\small{$j$}}\put(94,35){\small{$n$}}
\put(148,35){\small{$1$}}\put(184,35){\small{$i$}}
\put(212,35){\small{$j$}}\put(244,35){\small{$n$}}
\put(298,35){\small{$1$}}\put(334,35){\small{$i$}}
\put(362,35){\small{$j$}}\put(394,35){\small{$n$}}
\put(0,0){\circle*{5}}\put(36,0){\circle*{5}}
\put(64,0){\circle*{5}}\put(96,0){\circle*{5}}
\put(150,0){\circle*{5}}\put(186,0){\circle*{5}}
\put(214,0){\circle*{5}}\put(246,0){\circle*{5}}
\put(300,0){\circle*{5}}\put(336,0){\circle*{5}}
\put(364,0){\circle*{5}}\put(396,0){\circle*{5}}
\put(0,20){\circle*{5}}\put(36,20){\circle*{5}}
\put(64,20){\circle*{5}}\put(96,20){\circle*{5}}
\put(150,20){\circle*{5}}\put(186,20){\circle*{5}}
\put(214,20){\circle*{5}}\put(246,20){\circle*{5}}
\put(300,20){\circle*{5}}\put(336,20){\circle*{5}}
\put(364,20){\circle*{5}}\put(396,20){\circle*{5}}
\put(-35,10){\small{$\mathfrak{u}_{i}\mathfrak{u}_j=$}}
\put(118,10){\small{$\mathfrak{v}_{ij}=$}}
\put(268,10){\small{$\mathfrak{h}_{ij}=$}} \put(312,10){$\cdots$}
 \put(344,10){$\cdots$}\put(372,10){$\cdots$}
 \put(12,10){$\cdots$}
 \put(44,10){$\cdots$}\put(72,10){$\cdots$}
 \put(162,10){$\cdots$}
 \put(194,10){$\cdots$}\put(222,10){$\cdots$}
\qbezier[40](0,20)(0,10)(0,0) \qbezier[40](96,20)(96,10)(96,0)
\qbezier[40](150,20)(150,10)(150,0)
\qbezier[40](246,20)(246,10)(246,0)
\qbezier[40](300,20)(300,10)(300,0)
\qbezier[40](396,20)(396,10)(396,0)
\qbezier(36,0)(28,4)(36,8)\qbezier(36,0)(44,4)(36,8)
\qbezier(36,20)(28,24)(36,28)\qbezier(36,20)(44,24)(36,28)
\qbezier(64,0)(56,4)(64,8)\qbezier(64,0)(72,4)(64,8)
\qbezier(64,20)(56,24)(64,28)\qbezier(64,20)(72,24)(64,28)
\qbezier(186,20)(200,40)(214,20)\qbezier(186,0)(178,4)(186,8)
\qbezier(186,0)(194,4)(186,8) \qbezier(336,20)(350,40)(364,20)
\qbezier(336,0)(350,20)(364,0) \qbezier(214,0)(206,4)(214,8)
\qbezier(214,0)(222,4)(214,8)
\end{picture}
\end{center}
As for the action of $U_2$, a routine computation yields
$U_2\cdot \mathfrak{a}_{i,j}^{\cap}= x\,\mathfrak{a}_{i,j}^{\circ}$
and
$U_2\cdot \mathfrak{a}_{i,j}^{\circ} = x^2\,\mathfrak{a}_{i,j}^{\circ}$.
Similarly we obtain for $V_2$,
$V_2\cdot \mathfrak{a}_{i,j}^{\cap}= x\,\mathfrak{a}_{i,j}^{\cap}$ and
$V_2\cdot \mathfrak{a}_{i,j}^{\circ} = x^2\,\mathfrak{a}_{i,j}^{\cap}$.
Let $\tau_{(i,j)}$ act on the diagram $\mathfrak{a}_{i,j}^\cap$ as the
transposition $(i,j)\in S_n$ from the left and $\tilde{\tau}_{(a,b)}$
as transposition $(a,b)\in S_{n-2}$, from the right, respectively.
Then
\begin{eqnarray}
\label{E:oho3}
H_2\cdot \mathfrak{a}_{i,j}^{\circ} &=&  x\, \mathfrak{a}_{i,j}^\cap\\
\label{E:oho4}
H_2\cdot \mathfrak{a}_{i,j}^\cap &=&
\left((x-1) +\sum_{i<j}\tau_{(i,j)}-\sum_{a<b}\tilde{\tau}_{(a,b)}\right)\,
\mathfrak{a}_{i,j}^\cap,
\end{eqnarray}
where eq.~(\ref{E:oho4}) holds according to \cite{Doran}, Lemma $2$,
p.$655$. We write an element $v\in \mathscr{N}_{\mathbb{A}_n}(\mu)$
as
\begin{equation*}
v =\sum_{i,j}\mathfrak{a}_{i,j}^{
\cap}\otimes r_{i,j}+\sum_{i,j}\mathfrak{a}_{i,j}^{\circ}\otimes s_{i,j}
\end{equation*}
and set
$v^\cap=\sum_{i,j}\mathfrak{a}_{i,j}^{\cap}\otimes r_{i,j}$ and
$v^{\circ}=\sum_{i,j}\mathfrak{a}_{i,j}^{\circ}\otimes s_{i,j}$.
Since $(H_2-x^{-1}V_2)\in\mathbb{I}_n^{n-2}$, we obtain
\begin{eqnarray*}
(H_2-x^{-1}V_2)\cdot (v^\cap +v^{\circ}) & = & H_2\cdot v^\cap+
H_2\cdot v^{\circ}
-x^{-1}V_2\cdot v^\cap-x^{-1}V_2\cdot v^{\circ}\\
& = & H_2\cdot v^\cap + xg(v^{\circ}) - v^\cap -x \,g(v^{\circ})\\
& = & H_2\cdot v^\cap-v^\cap.
\end{eqnarray*}
Suppose now there exists some $0\neq v_0\in \varphi_n(S^\lambda)\subset
\mathscr{N}_{\mathbb{A}_n}(\mu)$ such that $v_0^\cap\neq 0$ and
$v_0^{\circ}\neq 0$.
Since $\varphi_n(S^\lambda)$ is an irreducible $S_n$-module and
the $S_n$-action cannot change a horizontal arc into a pair of loops,
for any $0\neq v\in \varphi_n(S^\lambda)\subset
\mathscr{N}_{\mathbb{A}_n}(\mu)$,
$v^\cap\neq 0$ and $v^{\circ}\neq 0$ holds.
Therefore if there exits some $0\neq v_0\in \varphi_n(S^\lambda)\subset
\mathscr{N}_{\mathbb{A}_n}(\mu)$ such that $v_0^\cap\neq 0$ and
$v_0^{\circ}\neq 0$, then we have for any $0\neq v\in  \varphi_n(S^\lambda)$,
$(H_2-1)\cdot v^\cap=0$, i.e.
\begin{equation}\label{E:allgm}
\left((x-1) +\sum_{i<j}\tau_{(i,j)}-\sum_{a<b}\tilde{\tau}_{(a,b)}
-1\right) \cdot v^\cap=0.
\end{equation}
We proceed by studying the action of $\sum_{i<j}\tau_{(i,j)}$ and
$\sum_{a<b}\tilde{\tau}_{(a,b)}$ on the set
\begin{equation}
\varphi_n^\cap(S^\lambda)=\{ v^\cap\mid v^\cap+v^{\circ}\in
\varphi(S^\lambda)\}.
\end{equation}
The $\mathbb{A}_n$-module $\mathscr{M}_{\mathbb{A}_n}(\mu)$ can be
regarded as a $S_n\times S_{n-2}$-left module via
\begin{equation}\label{E:act}
(\sigma,\sigma')\cdot (\mathfrak{a}\otimes w)=
\sigma\cdot(\mathfrak{a}\otimes \sigma'w)
\end{equation}
and $\sigma\cdot(\mathfrak{a}\otimes \sigma'w)=(\sigma\mathfrak{a}\sigma')
\otimes w$ shows that the action of eq.~(\ref{E:oho4}) and eq.~(\ref{E:act})
coincide.
Furthermore, $\varphi_n(S^\lambda)$ becomes via eq.~(\ref{E:act}) a 
$S_n\times S_{n-2}$-submodule of $\mathscr{M}_{\mathbb{A}_n}(\mu)$ and 
induces an $S_n\times S_{n-2}$ action on the set 
$\varphi_n^\cap(S^\lambda)$ via $(\sigma,\sigma')\cdot 
(\mathfrak{a}_{i,j}^{\cap}\otimes w)=
\sigma\cdot(\mathfrak{a}_{i,j}^\cap\otimes \sigma'w)$.
Accordingly, $\varphi_n^\cap(S^\lambda)$ can be considered as a
$S_n\times S_{n-2}$-module and the projection
\begin{equation}\label{E:proj1}
\pi_1\colon \varphi_n(S^\lambda)\longrightarrow \varphi_n^\cap(S^\lambda)
,\quad (v^\cap+v^{\circ})\mapsto v^\cap,
\end{equation}
establishes an isomorphism of $S_n\times S_{n-2}$-modules. Indeed, only 
injectivity needs to be proved. Using $x\neq 0$, $U_2\in\mathbb{I}_n^{n-2}$ 
and $(v^\cap +v^{\circ})\in\mathscr{N}_{\mathbb{A}_n}(\lambda)$,
injectivity follows from
\begin{equation*}
x^{-1}U_2\cdot (v^\cap +v^{\circ})  = g(v^\cap) + x v^{\circ} =0,
\end{equation*}
i.e.~$v^{\circ}=-x^{-1}g(v^\cap)$. Obviously, 
$\sum_{i<j}\tau_{(i,j)}$ and $\sum_{a<b}\tilde{\tau}_{(a,b)}$ are contained 
in the centers of the group algebras $F[S_n]$ and $F[S_{n-2}]$, respectively 
and Schur's Lemma implies that they act as homotheties on irreducible 
representations. Since $\varphi_n(S^\lambda)$ embeds into the 
$S^\lambda\otimes S^\mu$-component of $\mathscr{M}_{\mathbb{A}_n}(\mu)$,
the particular values of $\sum_{i<j}\tau_{(i,j)}$ and 
$\sum_{a<b}\tilde{\tau}_{(a,b)}$ are given by \cite{Sagan}
\begin{equation}
\sum_{i<j}\tau_{(i,j)} = \sum_{p\in[\lambda]}c(p)\quad\text{\rm and}\quad
\sum_{a<b}\tilde{\tau}_{(a,b)} = \sum_{p\in[\mu]}c(p).
\end{equation}
Since $\varphi_n(S^\lambda)\subset \mathscr{N}_{\mathbb{A}_n}(\mu)$ 
we obtain 
\begin{equation}
\forall\; v^\cap\in \varphi_n^\cap(S^\lambda);\qquad
\left((x-1) +\sum_{p\in[\lambda]}c(p)-\sum_{p\in[\mu]}c(p)
-1\right)v^\cap =0,
\end{equation}
which implies
\begin{equation}\label{E:voila}
(x-1) +\sum_{p\in[\lambda]}c(p)-\sum_{p\in[\mu]}c(p)
-1=0.
\end{equation}
Since the content $c(p)$ is an integer, eq.~(\ref{E:voila}) implies
$x\in \mathbb{Z}$.
It thus remains to consider the cases $v^\cap=0$ or $v^\circ=0$.
The case of $v^\circ=0$ is due to \cite{Doran}. In analogy we derive, 
using the action of $H_2$ on $\varphi_n(S^\lambda)$
\begin{equation}
\forall\; v\in \varphi_n(S^\lambda);\qquad
H_2\cdot v=\left((x-1) +\sum_{p\in[\lambda]}c(p)-\sum_{p\in[\mu]}c(p)
\right)\cdot v=0,
\end{equation}
which implies
$(x-1) +\sum_{p\in[\lambda]}c(p)-\sum_{p\in[\mu]}c(p)=0$. This immediatly
allows us to conclude $x\in \mathbb{Z}$. In case of $v^\cap=0$ we
obtain for any $v\in \varphi_n(S^\lambda)$
\begin{equation}
U_2\cdot v=x^2v =0,
\end{equation}
which is, in view of $x\neq 0$ impossible.\\
We have therefore showed that in case of $\mu\vdash (n-1)$, $x\neq 0$ implies
$\mathscr{N}_{\mathbb{A}_n}(\mu)=0$. Since
$\mathscr{N}_{\mathbb{A}_n}(\mu)$ is the unique, maximal
$\mathscr{M}_{\mathbb{A}_n}(\lambda)$-submodule, there cannot exist
an embedding $\varphi_n\colon S^\lambda\longrightarrow
\mathscr{M}_{\mathbb{A}_n}(\mu)$. In case of $\mu\vdash (n-2)$, our
proof guarantees that for $x\not\in\mathbb{Z}$, there exists no
embedding $\varphi_n\colon S^\lambda \longrightarrow
\mathscr{M}_{\mathbb{A}_n}(\mu)$, whence $\mathbb{A}_n$ is
semisimple.
\end{proof}

{\bf Acknowledgments.}
We are grateful to Andreas Dress and Jing Qin for helpful discussions.
This work was supported by the 973 Project, the PCSIRT Project of the
Ministry of Education, the Ministry of Science and Technology, and the
National Science Foundation of China.
\bibliography{a3}
\bibliographystyle{plain}

\end{document}